\documentclass[12pt]{article}

\usepackage[cp1251]{inputenc}
\usepackage{amsmath}
\usepackage{amsthm}
\usepackage{amsxtra}
\usepackage{amsfonts,amssymb}
\usepackage{epsfig}

\newtheorem{Th}{Theorem}
\newtheorem{Prop}[Th]{Proposition}
\newtheorem{Lm}[Th]{Lemma}
\newtheorem{Co}[Th]{Corollary}
\newtheorem{Conj}[Th]{Conjecture}

\def\be{\begin{eqnarray}} \def\ee{\end{eqnarray}} \def\bes{\begin{eqnarray*}}
\def\ees{\end{eqnarray*}} \def\bit{\begin{itemize}} \def\eit{\end{itemize}}
\def\ba{\begin{array}{ll}} \def\ea{\end{array}}

\def\mn{\mathbb{N}}

\def\mc{\mathbb{C}}  

\def\mr{\mathbb{R}}

\newtheorem{Def}[Th]{Definition}

 \def\C{\mathbb{C}} \def\sm{\setminus} \def\mn{\mathbb{N}}
\def\le{\leqslant} \def\ge{\geqslant} \def\ra{\rightarrow} 

\def\d{\text{dist}} \def\re{\text{Re}} \def\im{\text{Im}}

\def\scr{\scriptscriptstyle}
\author{Artem Dudko}
\title{Computability of the Julia set. Nonrecurrent critical orbits.}
\date{}
\begin{document}
\maketitle
\begin{abstract} We prove that the Julia set of a
rational function $f$ is computable in  polynomial time, assuming
that the postcritical set of $f$ does not contain any critical
points or parabolic periodic orbits.

\end{abstract}

\section{Introduction.}

A compact subset of the complex plane is called computable if it can
be visualized on a computer screen with an arbitrarily  high
precision. Computer-generated images of mathematical objects play an
important role in establishing new results. Among such images, Julia
sets of rational functions occupy one of the most prominent
position. Recently, it was shown that for a wide class of rational
functions their Julia sets can be computed efficiently (see
\cite{Bra04}, \cite{Bra06}, \cite{Ret05}) and yet some of those sets
are uncomputable, and so cannot be visualized (see \cite{BY06},
\cite{BBY09}). Also, there are examples of computable Julia sets
whose computational complexity is arbitrarily high (see
\cite{BBY06}).


One of the natural open questions of computational complexity of
Julia sets is how large is the class of rational functions (in a
sense of Lebesgue measure on the parameter space) whose Julia set
can be computed in a polynomial time. Informally speaking, such
Julia sets are {\it easy} to simulate numerically.\\ {\bf
Conjecture.} {\it The class of rational functions of degree
$d\geqslant 2$  whose Julia set can be computed in a polynomial time
has a full measure in the space of parameters.}\\ The main result of
the paper is the following. \\ {\bf Main Theorem.} {\it Let $f$ be a
rational function of degree $d\geqslant 2$. Assume that for each
critical point $c\in J_f$ the $\omega$-limit set $\omega(c)$ does
not contain either recurrent critical point or a parabolic periodic
point of $f$. Then the Julia set $J_f$ is computable in a polynomial
time.}\\ As the reader will see below, this result can be viewed as
a natural step towards a proof of the conjecture.

The paper is organized as follows. In Section $1$ we give all
necessary preliminaries in Computability and Complex Dynamics, state
the results of the paper and discuss possible generalizations. To
illustrate the results of the paper on a simple case in section $2$
we prove that for every subhyperbolic rational function the Julia
set is computable in a polynomial time. In Section $3$ we prove the
main result of the paper under a simplifying assumption that the
rational map $f$ does not have any parabolic periodic points. In
section $4$ we complete a proof of the main result.

{\bf Acknowledgements.} It is my great pleasure to thank my
supervisor Michael Yampolsky for posing the question of the paper
and numerous fruitful discussions.

\subsection{Preliminaries on computability}
In this section we give a brief introduction to computability and
complexity of functions and sets.
 The notion of computability relies on the concept of a Turing Machine (TM).
A precise definition of a Turing Machine is quite technical and we
do not give it here. For the definition and properties of a Turing
Machine we refer the reader to \cite{Pap} and \cite{Sip}. The
computational power of a Turing Machine is equivalent to that of a
 RAM computer with infinite memory. One can generally
 think about a Turing Machine as a formalized definition of an algorithm or a computer program.

There exist several different definitions of computability of sets.
For discussion of different approaches to computability we refer the
reader to \cite{BY08}. In this article we use the notion of
computability related to complexity of drawing pictures on a
computer screen. Roughly speaking, a subset $S$ is called computable
in time $t(n)$ if there is a computer program which takes time
$t(n)$ to decide whether to draw a given $2^{-n}\times 2^{-n}$
square pixel in a picture of $S$ on a computer screen, which is
accurate up to one pixel size. Before giving a rigorous definition
of a computable Julia set we need to introduce some notations.

First we give the classical
definitions of a computable function and a computable number.
\begin{Def}\label{comp fun def} Let $S,N$ be countable subsets of $\mathbb{N}$.
A function $f:S\rightarrow N$ is called computable if there exists a
TM which takes $x$ as an input and outputs $f(x)$.
\end{Def}
Note that Definition \ref{comp fun def} can be naturally extended to
functions on arbitrary countable sets, using a convenient
identification with $\mathbb{N}$.
\begin{Def} A real number $\alpha$ is called computable if there is a
 computable function $\phi:\mathbb{N}\rightarrow \mathbb{Q}$,
 such that for all $n$ $$\left|\alpha-\phi(n)\right|<2^{-n}.$$
 The set of computable reals is denoted by $\mathbb{R}_\mathcal{C}$.
\end{Def} In other words, $\alpha$ is called computable if there is an algorithm
 which can approximate $\alpha$ with any given precision.
The set $\mathbb{R}_\mathcal{C}$ is countable, since there are only
countably many algorithms. The set of computable complex numbers is
defined by
$\mathbb{C}_\mathcal{C}=\mathbb{R}_\mathcal{C}+i\mathbb{R}_\mathcal{C}$.
Note that both $\mathbb{R}_\mathcal{C}$ and $\mathbb{C}_\mathcal{C}$
considered with usual multiplication and addition form fields.
Moreover, it is easy to see that $\mathbb{C}_\mathcal{C}$ is
algebraically closed.


 Let $d(\cdot,\cdot)$ stand for Euclidian distance between points or sets in $\mathbb{R}^2$.
 Recall the definition of the Hausdorff distance between two sets:
$$d_H(S,T)=\inf\{r>0:S\subset U_r(T),\;T\subset U_r(S)\},$$
where $U_r(T)$ stands for the $r$-neighborhood of $T$:
$$U_r(T)=\{z\in \mathbb{R}^2:d(z,T)\leqslant r\}.$$ We call
a set $T$ a $2^{-n}$ approximation of a bounded set $S$ if $S\subset
T$
 and $d_H(S,T)\leqslant 2^{-n}$. When we try to draw a $2^{-n}$ approximation $T$ of a set $S$
 using a computer program, it is convenient to
 let $T$ be a finite
 collection of disks of radius $2^{-n-2}$ centered at points of the form $(i/2^{n+2},j/2^{n+2})$
 for $i,j\in \mathbb{Z}$.  Such $T$ can be described using a function
\begin{eqnarray}\label{comp fun}h_S(n,z)=\left\{\begin{array}{ll}1,&\text{if}\;\;d(z,S)\leqslant 2^{-n-2},
\\0,&\text{if}\;\;d(z,S)\geqslant 2\cdot 2^{-n-2},\\
0\;\text{or}\;1&\text{otherwise},
\end{array}\right.\end{eqnarray} where $n\in \mathbb{N}$ and $z=(i/2^{n+2},j/2^{n+2}),\;i,j\in \mathbb{Z}.$\\
\begin{figure}\centering\epsfig{file=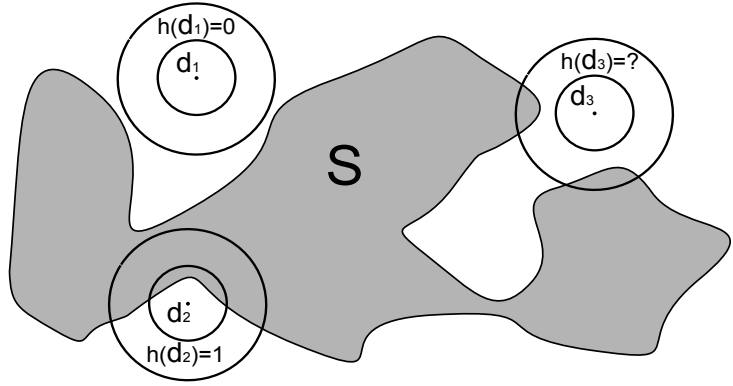,width=.70\linewidth}\caption{Values
of the function $h_S$.} \end{figure} \\ Using this
function, we define computability and computational
complexity of a set in $\mathbb{R}^2$ in the following way.
\begin{Def}\label{DefComputeSet} A bounded set $S\subset
\mathbb{R}^2$ is called computable in time $t(n)$ if there is a TM,
which computes values of a function $h(n,\bullet)$ of the form
(\ref{comp fun}) in time $t(n)$. We say that $S$ is poly-time
computable if there exists a polynomial $p(n)$, such that $S$ is
computable in time $p(n)$.
\end{Def}  Similarly, one can define computability and
computational complexity of subsets of $\mr^k$. Moreover, definition
\ref{DefComputeSet} naturally extends to subsets of
$\hat{\mathbb{C}}$ (see \cite{BY08} Section 2.1). The Riemann sphere
$\hat{\mathbb{C}}$ is homeomorphic to the unit sphere
$$S^2=\{x:|x|=1\}\subset \mathbb{R}^3.$$
 Consider the stereographic projection $$P:S^2\setminus
 \{\text{North Pole}\}\rightarrow \mathbb{C}.$$ The inverse
 of this projection is given by $$P^{-1}:z\rightarrow
 \left(\frac{2\re(z)}{|z|^2+1},\frac{2\im(z)}{|z|^2+1},
 \frac{|z|^2-1}{|z|^2+1}\right);\;P^{-1}(\infty)=(0,0,1).$$
Observe that $P^{-1}$ induces the spherical metric on
$\hat{\mc}$, given by the formula
$$ds=\frac{dz}{1+|z|^2}.$$ \begin{Def} A subset
$K\subset\hat{\mc}$ is called computable in time $t(n)$ if
$P^{-1}(K)\subset\mr^3$ is computable in time $t(n)$.
\end{Def} \begin{Prop} Let $K\subset \mc$ be a bounded
subset. Then $K$ is computable as a subset of $\hat{\mc}$
if and only if it is computable as a subset of $\mr^2$.
Similarly, $K$ is poly-time computable as a subset of
$\hat{\mc}$ if and only if it is poly-time computable as a
subset of $\mr^2$.\end{Prop}

In this paper we discuss computability of the Julia sets of
rational functions. For simplicity, consider the case of
quadratic polynomials $f_c(z)=z^2+c$. By computing the
Julia set $J_c$ of the map $f_c,\;c\in\mathbb{C},$ we mean
the following problem:
 $$\text{given the parameter}\;\;c\;\;\text{compute a
 function}\;\;
h\;\;\text{of the form}\;\;(\ref{comp
fun})\;\;\text{for}\;\;S=J_c.$$ However, an algorithm $M$, computing
$h$, can handle only a finite amount of information. In particular,
it can not read or store the entire input $c$ if $c\notin J_f$.
Instead, it may request this input with an arbitrary high precision.
In other words, the machine $M$ has a command $READ(m)$ which for
any integer $m$ requests the real and the imaginary part of a number
$\phi(m)$, such that $|\phi(m)-c|<2^{-m}$. It can be formalized
using the notion of an {\it oracle}. Let
$$\mathbb{D}=\left\{\frac{k}{2^l}:k\in
\mathbb{Z},l\in\mathbb{N}\right\}$$
  be the set of all dyadic numbers. Denote $$\mathbb{D}^n=\{(x_1,x_2,\ldots,x_n):x_j\in \mathbb{D},\;\;1\leqslant \
j\leqslant n\}$$ the $n$-th
  Cartesian power of $\mathbb{D}$. We will say that the open disk $U_d(c)$ (respectively closed disk
 $U_d(c)$) is
dyadic if $d$ and $c$ are dyadic. Denote $\mathcal{C}$ the set of all subsets $U\subset \mathbb{C}$
such that $U$ can be represented as a finite union of dyadic disks.
 \begin{Def} A function $\phi:\mathbb{N}\rightarrow\mathbb{D}^n$ is called an oracle for an element $x\in
\mathbb{R}^n$ if $\|\phi(m)-x\|<2^{-m}$ for all $m\in \mathbb{N}$,
where $\|\cdot\|$ stands for the Euclidian norm in $\mathbb{R}^n$.
 An {\it oracle Turing Machine} $M^\phi$ is a TM, which can query $\phi(m)$ for any $m\in \mathbb{N}$.\end{Def}
 The oracle $\phi$ is not a part of the algorithm, but rather enters as a parameter.
 In case of a computer program the role of the oracle is usually played by the user,
 who enters the parameters of the program.
 We should note also that an algorithm using an oracle $\phi$ require
 $m$ time units to read $\phi(m)$.

 Now we are ready to define computability and computational complexity of the Julia set of a rational map.
\begin{Def} Let $f$ be a rational map. The Julia set $J_f$ is called computable in time $t(n)$ if there is a Turing Machine with an oracle for the coefficients of $f$, which computes values of a function $h(n,\bullet)$ of the form
(\ref{comp fun}) for $S=J_f$ in time $t(n)$. We say that $J_f$ is
poly-time computable if there exists a polynomial $p(n)$, such that
$J_f$ is computable in time $p(n)$.
 \end{Def}
\subsection{Hyperbolic maps}
A rational map $f$ is called
 hyperbolic if there is a Riemannian metric $\mu$ on a neighborhood of the Julia set $J_f$ in which $f$ is strictly
expanding:
 $$\|Df_z(v)\|_\mu>\|v\|_\mu $$ for any $z\in J_f$ and any tangent vector $v$ (see \cite{M}). It follows that
 for a hyperbolic map $f$ there is a neighborhood $U$ of $J_f$
 on which the metric $d_\mu$ induced by $\mu$ is strictly expanding:
 $$d_\mu(f(x),f(y))\geqslant kd_\mu(x,y),\;\;\text{for any}\;\;x,y\in U,\;\;\text{where}\;\;k>1.$$
 Let $d_S$ be the spherical metric on the Riemann sphere.
   Then $d_\mu$ and $d_S$ are equivalent on a compact neighborhood of $J_f$.
   One can deduce that there exists $N\in \mathbb{N}$, $K>1$ and a neighborhood $V$ of $J_f$, such that
   $$d_S\left(f^N(x),f^N(y)\right)\geqslant Kd_S(x,y),\;\;\text{for any}\;\;x,y\in V. $$
   Moreover, the last property is equivalent to the definition of hyperbolicity.
   Hyperbolic maps have the following topological characterization (see \cite{M}).
 \begin{Prop} A rational map $f$ is hyperbolic if and only if every critical orbit of
 $f$ either converges to an attracting (or a
 super-attracting) cycle, or is periodic.
 \end{Prop}
 Braverman in \cite{Bra04} and Rettinger in
 \cite{Ret05} have independently proven the following result.
 \begin{Th} For any $d\geqslant 2$ there exists a Turing Machine with an oracle for the coefficients of a rational map of degree $d$ which computes the Julia set of every hyperbolic rational map in polynomial time.
 \end{Th}
 To explain why this is an important result we would like to mention the following.
 It is known that hyperbolicity is an open condition in the space of coefficients
 of rational maps of degree $d\geqslant 2$. The famous conjecture of Fatou conjecture states
  that the set of hyperbolic parameters is dense in this space.
  This conjecture is known as the Density of Hyperbolicity Conjecture.
  It is the central open question in Complex Dynamics.
  Lyubich \cite{L} and Graczyk-Swiatek \cite{GS} independently
  showed that this conjecture is true for the real quadratic family.
  Namely, the set of real parameters $c$, for which the map $f(z)=z^2+c$
  is hyperbolic, is dense in $\mathbb{R}$.

\subsection{Subhyperbolic maps}
A rational map $f$ is called subhyperbolic if
 it is expanding on a neighborhood of the Julia set $J_f$
 in some orbifold metric (see \cite{M}). An orbifold metric is a conformal metric
 $\gamma(z)dz$ with a finite number of singularities of the following form.
 For each singularity $a$ there exists an
 integer index $\nu=\nu_a>1$, such that for the branched covering $z(w)=a+w^\nu$ the induced metric
 $$\gamma(z(w))\left|\frac{dz}{dw}\right|dw$$ in $w$-plane is smooth and nonsingular
 in a neighborhood of the origin. Douady and Hubbard proved the
 following (see \cite{M})
  \begin{Prop}\label{subhyp crit} A rational map is subhyperbolic,
  if and only if every critical orbit of $f$ is either finite or converges to
  an attracting (or a super-attracting) cycle.\end{Prop} For a subhyperbolic map,
  the corresponding orbifold has singularities only at postcritical points, which lie
 in $J_f$. In presence of these singularities the algorithm, which works in the case of
 hyperbolic maps, can not be applied directly for subhyperbolic maps. We show in this paper how to change this algorithm to compute the Julia set of a subhyperbolic map in a polynomial time. One of the results of this paper is the following theorem.
 \begin{Th}\label{subhyp res} There is a TM $M^\phi$ with an oracle for the coefficients of a rational map $f$, such that for every
 subhyperbolic map $f$ the machine $M^\phi$
 computes $J_f$ in polynomial time, given some finite non-uniform information about the orbits of critical points of $f$.
 \end{Th}  We will specify later, which non-uniform information does $M^\phi$ use.

\subsection{The main results: nonrecurrent critical orbits.}
Let $f$ be a rational map. For a point $z\in \mathbb{C}$ denote
$$\mathcal{O}(z)=\{f^n(z):n\geqslant 1\}$$ the
 forward orbit of  $z$. Let $\omega(z)=\overline{\mathcal{O}(z)}
 \setminus \mathcal{O}(z)$ be the $\omega$-limit set of $z$. Denote $C_f$
 the set of critical points of $f$ which lie in $J_f$. Put
 $$\Omega_f=\bigcup\limits_{c\in C_f}\omega(c).$$
Our main result is devoted to the class of rational maps without recurrent critical points. Namely, we prove here the following:
\begin{Th}\label{main} Let $f$ be a rational map. Assume that $f$ has
 no parabolic periodic points and $\Omega_f$ does not
 contain any critical points. Then $J_f$ is poly-time computable by a TM with an oracle for the coefficients of $f$.
\end{Th} Although in this case there is no guarantee that any kind of expansion holds on the whole Julia set of $f$,
the following well-known theorem of Ma\~{n}\'{e}  (see \cite{Mane},\cite{Shi}) implies that there is expansion on the closure of
the postcritical set of $f$.
\begin{Th}\label{Mane thm} Let $f$ be a rational map. Let $M\subset J_f$ be a compact invariant  set, such that $M$
does not contain any critical points of $f$ or parabolic periodic points and $M\cap \omega(c)=\varnothing$
 for any recurrent critical point of $f$. Then there exists $N\in \mathbb{N}$,
 such that $\left|Df^N(z)\right|>1$ for any $z\in M$.
\end{Th}
As we see, the theorem of Ma\~{n}\'{e}  gives expansion only near the points, whose forward orbits are isolated from the
 parabolic periodic points.
The question arises: is it possible to generalize Theorem \ref{main} for the maps $f$ with parabolic periodic points?
It is known that the algorithms, used for computing hyperbolic Julia sets,
 require exponential time in the presence of parabolic points (see \cite{M}, app. H).
  However, in the paper \cite{Bra06} (see also \cite{BY08})
 Braverman proved that for any rational function $f$, such that every critical orbit of $f$
converges either to an attracting or to a parabolic orbit, the Julia set is poly-time computable.
Combining the algorithm used in \cite{Bra06} with the algorithm, which we use in the present paper
to prove Theorem \ref{main}, we obtain the following generalization of Theorem \ref{main}:
\begin{Th}\label{main1} Let $f$ be a rational map. Assume that $\Omega_f$ does not
 contain neither critical points nor parabolic periodic points.
 Then $J_f$ is poly-time computable by a TM with an oracle for the coefficients of $f$.
\end{Th}

\subsection{Possible generalizations.}
In this subsection we discuss possible generalizations of the results of this paper.
First we would like to mention that the algorithm which we use to prove Theorem \ref{main1} cannot be applied to compute the Julia set of a
 rational map $f$, such that
 $\Omega_f$ contains a parabolic periodic point. However, we believe that the following statement is true.
\begin{Conj} Let $f$ be a rational map. Assume that $\Omega_f$ does not
 contain any critical points.
 Then $J_f$ is poly-time computable by a TM with an oracle for the coefficients of $f$.
\end{Conj}
Another important class of rational maps is the class of Collet-Eckmann maps.
\begin{Def} Let $f$ be a rational map. Assume that there exist constants $C,\gamma>0$
such that the following holds:
 for any critical point $c$ of $f$ whose forward orbit
 does not contain any critical points one has:
 \begin{eqnarray}\label{CE} \left|Df^n(f(c))\right|\geqslant Ce^{\gamma n}\;\;\text{for any}\;\;n\in \mathbb{N}.
\end{eqnarray} Then we say that the map $f$ is Collet-Eckmann (CE). The condition $(\ref{CE})$ is called the Collet-Eckmann condition.
\end{Def}
In \cite{AM} Avila and Moreira showed that for almost every real parameter $c$ the map $f_c(z)=z^2+c$
 is either Collet-Eckmann or hyperbolic.
In \cite{Asp} Aspenberg proved that the set of Collet-Eckmann maps
has positive Lebesgue measure in the parameter space of all rational
maps of fixed degree $d\geqslant 2$. Moreover, there is a conjecture
that for almost all rational maps $f$ in this space the following is
true:
\begin{itemize}
\item[1)] the forward orbit of every critical point $c\notin J_f$
either is finite or converges
to an attracting periodic orbit;
\item[2)] for any critical point $c\in J_f$ either there exists $\gamma,C>0$
such that $\left|Df^n(f(c))\right|\geqslant Ce^{\gamma n}$ for any
$n\in \mathbb{N}$ or the forward orbit of $c$ contains another
critical point.
\end{itemize}

By the theorem of Ma\~{n}\'{e} \ref{Mane thm}, every rational map $f$ such that
$\Omega_f$ does not contain neither critical
 points nor parabolic periodic points is Collet-Eckmann.
 We believe that the following generalization of Theorem \ref{main1} is true.
 \begin{Conj}\label{Conj 2}
 Let $f$ be a rational map such that the conditions $1)$ and $2)$ above are satisfied.
 Then $J_f$ is poly-time computable by a TM with an oracle for the coefficients of $f$.
 \end{Conj}
Thus, we conjecture that \begin{Conj} For almost all rational maps
$f$ of degree $d\geqslant 2$ $J_f$ is poly-time computable.
 \end{Conj}

\section{Poly-time computability for subhyperbolic maps.}
In this section we prove Theorem \ref{subhyp res}. Namely,
we construct an algorithm $A$ which for every subhyperbolic
rational map $f$ computes $J_f$ in polynomial time. The
algorithm $A$ uses the coefficients of the map $f$ and some
non-uniform information which we will specify in the
following subsection. \subsection{Preparatory steps and
non-uniform information.} In this paper we actively use the
classical Koebe distortion theorem (see \cite{Con}). Let us
state it here. For $\delta>0,z\in \mathbb{C}$ denote
$$U_\delta(z)=\{w\in \mathbb{C}:|w-z|<\delta\}.$$
\begin{Th}\label{Koebe thm} Let $f:U_r(a)\rightarrow
\mathbb{C}$ be a univalent function. Then for any $z\in
U_r(a)$ one has:
\begin{eqnarray}\label{Koebe1}\frac{(1-|z-a|/r)|f'(a)|}{(1+|z-a|/r)^3}
\leqslant|f'(z)|\leqslant
\frac{(1+|z-a/r)|f'(a)|}{(1-|z-a|/r)^3} \\
\label{Koebe2}\frac{|z-a||f'(a)|}{(1+|z-a|/r)^2}
\leqslant|f(z)-f(a)|\leqslant
\frac{|z-a||f'(a)|}{(1-|z-a|/r)^2}\end{eqnarray} \end{Th}
The statement (\ref{Koebe2}) of the Koebe distortion
theorem can be reformulated the following way. Let
$1>\alpha>0$, $r_1=\frac{\alpha|f'(a)|r}{(1+\alpha)^2}$,
$r_2=\frac{\alpha|f'(a)|r}{(1-\alpha)^2}$. Then
\begin{eqnarray}\label{Koebe}U_{r_1}(a)\subset f(U_{\alpha
r}(a))\subset U_{r_2}(a).\end{eqnarray} We will also use
Koebe One-Quarter Theorem, which can be derived from Koebe
Distortion Theorem.\begin{Th}\label{Koebe quater} Suppose
$f:U_r(z)\ra \C$ is a univalent function. Then the image
$f(U_r(z))$ contains the disk of radius
${\scriptscriptstyle\frac{1}{4}}r|f'(z)|$ centered at
$f(z)$. \end{Th} Fatou and Julia proved the following
fundamental result. \begin{Th}\label{crit in attr} Let $f$
be a rational map of degree $d\ge 2$. Then the immediate
basin of each attracting cycle contains at least one
critical point. In particular, the number of attracting
periodic orbits is finite and do not exceed the number of
critical points.\end{Th}

 In our proof of Theorem \ref{subhyp
res} we will use the following general fact (see e.g. see \cite{Wey}). \begin{Prop}\label{root
finding} Let $h(z)$ be a complex polynomial. There exists a TM $M^\phi$ with an oracle for the
coefficients of $h(z)$ and a natural number $n$ as an input,  such that $M^\phi$ outputs a finite
sequence of complex dyadic numbers $\beta_1,\beta_2,\ldots, \beta_k$ for which:\\$1)$ each
$\beta_i$ lies at a distance not more than $2^{-n}$ from some root of $h(z)$;\\$2)$ each root of
$h(z)$ lies at a distance not more than $2^{-n}$ from one of $\beta_i$. \end{Prop}
Denote $N_F(f)$ and $N_J(f)$ the number of critical points of $f$ which lie in the Fatou set and
the Julia set of $f$ correspondingly. The algorithm computing the Julia set of a
 subhyperbolic map will use the numbers \begin{eqnarray}\label{non uniform info}N_F(f),\;N_J(f)
 \end{eqnarray} as the non-uniform information.

 Observe
 that the assertions $N_F(f)=0$ and $J_f=\hat{\mc}$ are
 equivalent. Indeed, by Proposition \ref{subhyp crit} if
 $N_F(f)>0$ then there is at least one attracting (or
 superattracting) periodic point and the Fatou set is
 nonempty. On the other hand if the Fatou set is nonempty,
 then there is at least one attracting (or
 superattracting) periodic point. The basin of this periodic point
 contains a critical point. Thus if $N_F(f)=0$, then the problem of computing $J_f$ is
trivial. Therefore, we will assume that $N_F(f)\neq 0$ and
$J_f\neq \hat{\mc}$.

Also, without loss of generality we may assume that $\infty\notin
J_f$. Indeed if $\infty\in J_f$, then using Proposition \ref{root
finding} we can find a dyadic point $z_0$ which lie in the Fatou set
of $f$. Let $h:\hat{\mc}\rightarrow\hat{\mc}$ be a M\"{o}bius map
 such that $h(\infty)=z_0$. Consider the map $h^{-1}\circ f\circ
 h$ instead of $f$.

 As a corollary of Proposition \ref{root finding} one can obtain the following result (see
\cite{BY08}, Proposition 3.3). \begin{Prop}\label{attr orbits
finding} Let $f$ be a subhyperbolic rational map. There exists a
Turing Machine $M^\phi$ with an oracle for the coefficients of $f$
such that the following is true. Given the number $N_F(f)$ of
critical points $c\notin J_f$ $M^\phi$ outputs a dyadic set $B$
such that \\$1)$ all the attracting and super attracting orbits of $f$ belong to $B$, \\$2)$ for
any $z\in B$ the orbit of $z$ converges to an attracting periodic orbit, \\$3)$ $f(B)\Subset B$.
\end{Prop} \begin{proof} The algorithm works as follows. Initially, let $\mathcal{B}$ be an empty
collection of dyadic disks. At $m$-th step, $m\in \mathbb{N}$, do the following. By Proposition
\ref{root finding} there
 is an algorithm which finds all periodic points
 of $f$ of period at most $m$ with precision $2^{-m-3}$.
 Let $p_i$ be approximate position of a periodic point of period
 $k_i$. The corresponding periodic point $z_i$ belongs to the disk
  $U_{2^{-m-3}}(p_i)$. Consider the disk $U_{2^{-m/2}}(p_i)$. If
$U_{2^{-m/2}}(p_i)$ does not intersect neither one of the
disks from $\mathcal{B}$, then
  approximate the image $f^{k_i}(U_{2^{-m/2}}(p_i))$
  by a dyadic set with precision $2^{-m-3}$. Namely, find a set
  $W_i\in \mathcal{C}$ such that
  $$d_H\left(W_i,f^{k_i}(U_{2^{-m/2}}(p_i))\right)< 2^{-m-1}.$$
    Verify if \begin{eqnarray}\label{orbit find1}U_{2^{-m}}(W_i)\subset
    U_{2^{-m/2}}(p_i).
    \end{eqnarray} This would imply that
    \begin{eqnarray}\label{orbit find2}
    f^{k_i}(U_{2^{-m/2}}(p_i))\subset U_{2^{-m-1}}(W_i)\subset
    U_{2^{-m/2}-2^{-m-1}}(p_i).
    \end{eqnarray}
    In this case, compute
 dyadic sets $B_j$, such that
 $B_0=U_{2^{-m}}(W_i)$ and for each
 $j=0,1,\ldots,k_i-1,\;f(B_j)\Subset B_{j+1}$, where $j+1$ is taken
 modulo $k_i$. Add dyadic sets $B_j$ to the collection $\mathcal{B}$.

Next, calculate
  approximations $s_i$ of the images
 $f^m(c_i)$ of critical points
 of $f$ such that $|s_i-f^m(c_i)|<2^{-m-1}$. If there are $N_F(f)$ of the points $s_i$,
  such that $$U_{2^{-m-1}}(s_i)\subset B=\bigcup\limits_{S\in \mathcal{B}}S,$$
   then we stop the algorithm and output
 $B$. Otherwise, go to step $m+1$.

Let us show that the algorithm eventually stops and outputs a set
$B$, satisfying the conditions $1)-3)$ of Proposition \ref{attr
orbits finding}. Let $z$ be an attracting (or super-attracting)
periodic point of period $k$ with multiplier $\lambda$. Let
    $|\lambda|<r<1$. Then for small enough $\varepsilon>0$ one has
    $$f^{k_i}(U_\varepsilon(z_i))\subset U_{r\varepsilon}(z_i).$$
It follows that for some $m>k$ the corresponding approximation $p_i$
of $z$ and the set $W_i$ satisfy the property (\ref{orbit find1}).
On the other hand if
    (\ref{orbit find2}) holds, then by Schwartz Lemma $U_{2^{-m/2}}(p_i)$
 contains an attracting periodic point, whose basin contains
 $U_{2^{-m/2}}(p_i)$. Therefore if the algorithm runs sufficient amount of steps,
  the union $B$ of dyadic sets from
 $\mathcal{B}$ satisfies the condition $1)$ of the Proposition
\ref{attr orbits finding}.

 By Proposition \ref{subhyp crit}, the orbit of each critical point of
 $f$ which does not lie in $J_f$ converges to an attracting periodic
 orbit. Thus, for some $m$ there will be $N_F(f)$ of the points $s_i$, belonging to $B$. This
 implies, then the algorithm stops.
  By Theorem \ref{crit in attr}, we obtain that $B$
   contains all attracting periodic orbits of
  $f$. Notice that conditions $2)$ and $3)$ of the Proposition \ref{attr orbits finding}
   for this set $B$ are satisfied by construction.
\end{proof}

 Denote $CF_f$ the set of critical points of $f$ which lie in the Fatou set of $f$. Put
$$PF_f=\bigcup\limits_{j\geqslant 0}f^j(CF_f).$$ The next
statement is a subhyperbolic analog of Proposition 3.7 from
\cite{BY08}. \begin{Prop}\label{set U} There exists an algorithm
which, given the coefficients of a subhyperbolic rational map $f$ of
degree $d\geqslant 2$ and the number $N_F(f)$, outputs a planar
domain $U\in \mathcal{C}$ such that:\begin{itemize} \item[(1)]
$U\Subset f(U)$, \item[(2)] $f(U)\cap PF_f=\varnothing$, \item[(3)]
$J_f\Subset U$.
\end{itemize} \end{Prop} \begin{proof} First use the
algorithm from Proposition (\ref{attr orbits finding}) to
find a dyadic set $B$ satisfying to the conditions $1)-3)$
of Proposition \ref{attr orbits finding}. Let $m\in
\mathbb{N}$. We can algorithmically construct a dyadic set
$W$ such that $$f^{-m}(B)\Subset W\Subset f^{1-m}(B).$$
Compute a dyadic number $d>0$ such that $$
U_d\left(W\right)\Subset f^{-1}(W).$$ Also, compute dyadic
approximations $s_i$ of critical points $c_i$ of $f$ such
that $$|s_i-c_i|\leqslant d.$$ By Theorem \ref{subhyp
crit}, for any critical point $c_i\in CF_f$ we will
eventually have: $$s_i\in U_d(c_i)\subset
U_d\left(W\right)\Subset f^{-1}(W).$$ Therefore, for large
enough $m$, the set $W$ will contain $N_F(f)$ of the points
$s_i$. Take such $m$. Then for any $s_i\in W$ one has
$$c_i\in U_d(s_i)\in U_d(W)\subset f^{-1}(W).$$ Thus,
$CF_f\subset f^{-1}(W)$. Compute a dyadic set
$\widetilde{W}$ such that $$f^{-3}(W)\subset
\widetilde{W}\subset f^{-2}(W).$$ Clearly, for the set
$U=\mathbb{C}\setminus \widetilde{W}$ conditions $(1)-(3)$
hold. \end{proof}

\subsection{Construction of the subhyperbolic metric.} Here we give the construction of a
subhyperbolic metric (see \cite{M}).  We modify the construction from \cite{M} to be able to write
it as an algorithm.
  First we recall the definition and basic properties of an orbifold.
  We refer the reader to \cite{M} for details.
\begin{Def}\label{orbifold} An orbifold $(S,\nu)$ is a Riemann surface together with a function
$\nu:S\rightarrow \mathbb{N}$ such that the set $\{z\in S:\nu(z)\neq 1\}$ is discrete. Points $z$
for which $\nu(z)\neq 1$ are called branch points. \end{Def} \begin{Def}\label{subhyp met def} Let
$f$ be a subhyperbolic rational map. An orbifold metric $\mu$ on a neighborhood $U$ of $J_f$ is
called subhyperbolic if $f$ is strictly expanding on $U$ with respect to $\mu$:
$$\|Df(z)\|_\mu>\lambda>1$$ for any $z\in f^{-1}(U)$
 except the branch points.
 \end{Def}
Let $p:S'\rightarrow S$ be a regular branched covering. Then for
every $z\in S$ the local degree of $p$ at a point $w\in p^{-1}(z)$
does not depend on $w$. One can define the weight function
$\nu:S\rightarrow\mathbb{N}$ of the covering $p$ assigning $\nu(z)$
the local degree of $p$ at $w\in f^{-1}(z)$.
\begin{Def} Let $(S,\nu)$ be an orbifold.
 A regular branched covering $$p:S'\rightarrow S$$ with the weight
 function $\nu$ such that $S'$ is simply connected is called a
 universal covering of the orbifold $(S,\nu)$. We will use the notation
 $\widetilde{S}_\nu\rightarrow (S,\nu)$ for a universal covering of this orbifold.
\end{Def} \begin{Prop}\label{uni cov} Let $(S,\nu)$ be an orbifold. The universal covering
$$\widetilde{S}_\nu\rightarrow (S,\nu)$$ exists and unique up to conformal isomorphism, except in
the following two cases:\\
$1)$ $S\thickapprox \widehat{\mathbb{C}}$ (the Riemann sphere) and
$S$ has only one branch point;\\
$2)$ $S\thickapprox \widehat{\mathbb{C}}$ and $S$ has two branch points $a_1,a_2$ such that
$\nu(a_1)\neq\nu(a_2)$. \end{Prop} The Euler characteristic of an orbifold $(S,\nu)$ is the number
$$\chi(S,\nu)=\chi(S)-\sum\limits_{z\in S}\left(1-\frac{1}{\nu(z)}\right).$$ Since the set of
branch points of $S$ is discrete the last sum contains at most countable number of nonzero terms.
If $S$ contains infinitely many branch points then we set $\chi(S,\nu)=-\infty$. The orbifold
$(S,\nu)$ is called hyperbolic if $\chi(S,\nu)<0$. \begin{Lm}\label{uni cov hyp} If $(S,\nu)$ is a
hyperbolic orbifold then $\widetilde{S}_\nu$ conformally isomorphic to the unit disk. \end{Lm} Let
$f$ be a subhyperbolic rational map. Let $U$ be the set from Proposition \ref{set U}. Construct an
orbifold $(U,\nu)$ in the following way. Put $S=U$. Denote $CJ_f$ the set of the critical points of
$f$ which lie in $J_f$. As the set of branch points of
 $U$ take $BP=\{f^j(c),c\in CJ_f,j\in \mathbb{N}\}$. Since
 $f$ is subhyperbolic, $BP$ is finite. Put $\nu(z)=1$ for all
 $z\in U\setminus BP$. Denote $n(f,z)$ the local degree of $f$
 at $z$. Define numbers $\nu(a),a\in BP$ such that the following
 condition holds:\begin{eqnarray}\label{cond for orb}
 \text{for any}\;\;z\in U\;\;\nu(f(z))\;\;\text{is a multiple of}
 \;\;\nu(z)n(f,z).\end{eqnarray}
 \begin{figure}\centering\epsfig{file=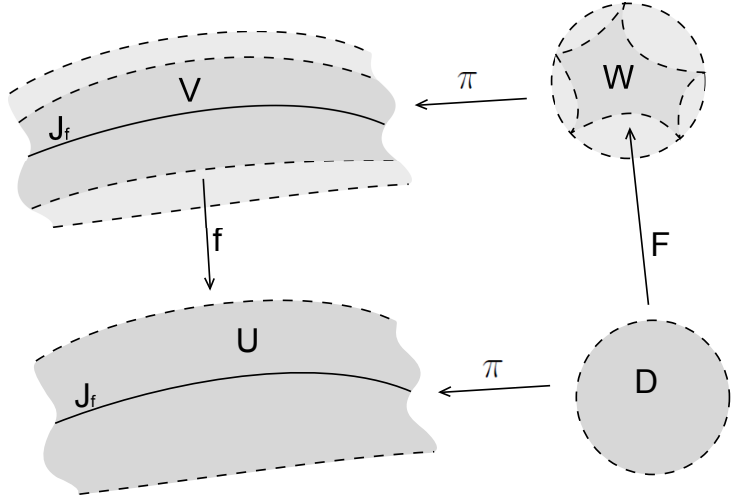,width=0.7\linewidth}
\caption{Construction of the subhyperbolic metric}\end{figure} If
the orbifold $(U,\nu)$ that we
  obtained is not hyperbolic, take any repelling orbit $\{z_j\}$
  of $f$ of the length at least $5$ and replace $\nu(z_j)$ with
  $2\nu(z_j)$ for all $z_j$ from this orbit. The new orbifold
  will be hyperbolic and satisfying
   the condition (\ref{cond for orb}).

By Proposition \ref{uni cov} there exists a universal covering
$$\widetilde{U}_\nu\overset{\pi}{\rightarrow}(U,\nu).$$ By Lemma \ref{uni cov hyp}, without loss of
generality we may assume that $\widetilde{U}_\nu=\mathbb{U}$ is the open unit disc. Since
$V=f^{-1}(U)\subset U$, Condition (\ref{cond for orb}) guarantee that the map $f^{-1}$ lifts to a
holomorphic map $$F:\mathbb{U}\rightarrow \mathbb{U}.$$ Note that $W=F(\mathbb{U})= \pi^{-1}(V)$ is
strictly contained in $\mathbb{U}$. Denote $\rho_\mathbb{U}$ the Poincar\'{e} metric on
$\mathbb{U}$. By the Schwartz-Pick Theorem, the map $F$ is strictly decreasing in the metric
$\rho_\mathbb{U}$. Let $\mu$ be the projection of $\rho_\mathbb{U}$ onto $U$. For any $w\in
\mathbb{U}$ and $z=\pi(F(w))$ one has: \begin{eqnarray}\label{Dfz}\|Df(z)\|_\mu=\|DF(w)\|^{-1}.
\end{eqnarray} It follows that the map $f$ is strictly expanding with respect to the norm induced
by $\mu$. To show that the metric $\mu$ is subhyperbolic we need to
prove that $\mu$ is uniformly strictly expanding. First we will
prove an auxiliary lemma.
 \begin{Lm}\label{sub and euclid}
    There exist a constant $C>0$ such that
   for any $z\in U$ except the branch points of the orbifold $(U,\nu)$ one has:
    $$C^{-1}<\left|\frac{\mu(z)}{dz}\right|<C\max \left\{|z-a_j|^{\frac{1}{\nu(a_j)}-1}\right\},$$
    where $a_j$ are the branch points.
   The constant $C$ can be obtained constructively.\end{Lm}
\begin{proof} Let $z\in U$. Assume that $z$ is not a branch point. To
estimate $\mu(z)$ without loss of generality we may assume that $z=\pi(0)$. Recall that the
 Poincar\'{e} metric on the unit disk is of the form:
 $$|\rho_\mathbb{U}(w)|=\frac{2|dw|}{1-|w|^2}.$$ Therefore one
has:$$\left|\frac{\mu(z)}{dz}\right|=\frac{2}{|D\pi(0)|}.$$ We can
construct a dyadic number $R$ such that $U$ is contained in a disk
of radius $R$. By Schwartz Lemma, $$|D\pi(0)|\leqslant R.$$ Thus,
$$\left|\frac{\mu(z)}{dz}\right|\geqslant 2R^{-1}.$$

 \begin{figure}\centering\epsfig{file=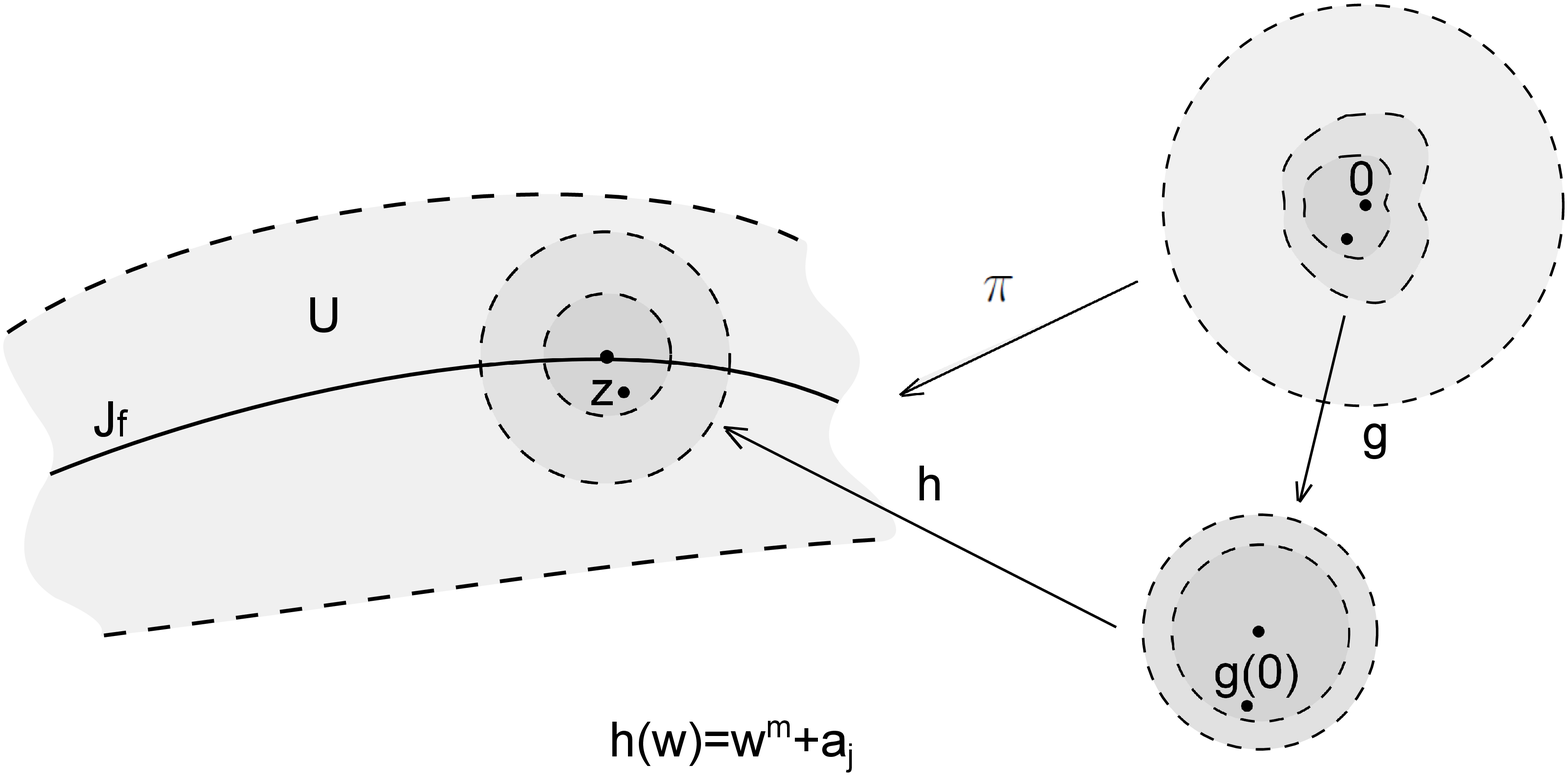,width=\linewidth}
\caption{Illustration to Lemma \ref{sub and
euclid}}\end{figure}
Recall that the set $BP$ of branch points
consists of critical points which lie in $J_f$ and possibly one
repelling periodic orbit. It follows from Proposition \ref{root
finding} that we can construct a dyadic number $\varepsilon>0$
 such that disks $U_{2\varepsilon}(a_j)$ are pairwise disjoint and all belong to $U$.
  Assume first that $z\notin \bigcup U_\varepsilon(a_j)$. Considering the branch
  $$\phi:U_\varepsilon(z)\rightarrow \mathbb{U},$$ of $\pi^{-1}$
   such that $\phi(z)=0$, by Schwartz Lemma
  we obtain $$|D\pi(0)|\geqslant \varepsilon.$$ Let $z\in U_\varepsilon(a_j)$ for some $j$.
   Then in a neighborhood $V(0)$ of $0$ the map
  $\pi$ can be written in the form $$\pi(w)=g(w)^m+a_j,$$ where $g(w)$ is a one to one
   map from $V(0)$ onto $U_\delta(0)$ with $\delta=(2\varepsilon)^{1/m}$.
   In $V(0)$ one has \begin{eqnarray}\label{D pi}D\pi(w)=mg(w)^{m-1}Dg(w).\end{eqnarray}
    Consider the map $\chi=g^{-1}:U_\delta(0)\rightarrow V(0).$
   Since $$|g(0)|=|z-a_j|^{1/m}\leqslant 2^{-1/m}\delta,$$ by Schwartz Lemma we get:
   $$|D\chi(g(0))|\leqslant K=K(\varepsilon),$$ where constant $K>0$ can be obtained
   constructively. Now (\ref{D pi}) implies $$ |D\pi(0)|\geqslant mK^{-1}|z-a_j|^{1-1/m},$$
   which finishes the proof.
   \end{proof}

\begin{Prop}\label{lambda} There exists a constant $\lambda>1$ such that $$\|Df(z)\|_\mu>\lambda$$
for any $z\in V=f^{-1}(U)$. The constant $\lambda$ can be constructed algorithmically. \end{Prop}
\begin{proof} For a map $g:U_1\rightarrow U_2$ between two hyperbolic Riemann surfaces denote
$$\|Dg(z)\|_{U_1,U_2}$$ the magnitude of the derivative of $g$ computed with respect to the two
Poincar\'{e} metrics. Denote $\d_\mu$ the distance in the metric $\mu$ on $U$ and $\d_\mathbb{U}$
the distance in the Poincar\'{e} metric on $\mathbb{U}$. Let $$i:W\rightarrow \mathbb{U}$$ be the
inclusion map. Then we have:
$$\|DF(w)\|_{\mathbb{U},\mathbb{U}}=\|DF(w)\|_{\mathbb{U},W}\|Di(F(w))\|_{W,\mathbb{U}}\leqslant
\|Di(F(w))\|_{W,\mathbb{U}}.$$ Using Lemma \ref{sub and euclid} it is not hard to show that we can
construct a dyadic constant $R>0$ such that $$\d_\mu(z,U\setminus V)<R\;\;\text{for any}\;\;z\in
V.$$ Let $z\in V$. Then there exists $ w \in W$ and $\zeta\in \mathbb{U}\setminus W$ such that
$$\pi( \omega )=z\;\;\text{and}\;\;\d_\mathbb{U}(\omega,\zeta)<R.$$ A suitable fractional linear
transformation sends $\zeta$ to $0$ and $\omega$ to $x>0$. Explicit calculations show that
\begin{eqnarray}\label{dx0}d=d_\mathbb{U}(x,0)=\log\frac{1+x}{1-x}, \;\;\text{so
that}\;\;x=\frac{e^d-1}{e^d+1}\leqslant
\frac{e^R-1}{e^R+1}.\end{eqnarray} Now, by Schwartz-Pick Theorem
\begin{eqnarray}\label{Diw}\|Di(\omega)\|_{W,\mathbb{U}}\leqslant\|Di(x)\|
_{\mathbb{U}\setminus\{0\},\mathbb{U}}.\end{eqnarray} The right hand
side of the last inequality can be estimated explicitly. It is equal
to $$a(x)=\frac{2|x\log x|}{1-x^2}<1.$$ Note that $a(x)$ increases
with $x$. By (\ref{dx0}) and (\ref{Diw}) the value
$\|Di(w)\|_{W,\mathbb{U}}$ is bounded from above with $a(X)<1$ for
$X=(e^R-1)/(e^R+1)$. By (\ref{Dfz}) we obtain that $\|Df(z)\|_\mu$
for $z\in V$ is bounded from below with $1/a(X)$. \end{proof}
\begin{Co} The metric $\mu$ is subhyperbolic.\end{Co} Lemma \ref{sub
and euclid} and Proposition \ref{lambda} together make a
subhyperbolic analog of Proposition 3.6 from \cite{BY08}.

\subsection{The algorithm.} Denote $V_k=f^{-k}(V)$. Notice that $$J_f\Subset V_{k+1}\Subset
V_k\Subset V$$ for any $k\in \mathbb{N}$. Let us prove the following auxiliary statement.
\begin{Prop}\label{using Koebe} There is an algorithm computing two dyadic constants $K_1,K_2>0$
 such that for any
$z\in V_3\setminus J_f$ and any $k\in \mathbb{N}$ if $f^k(z)\in V_1\setminus V_3$ then one has $$
\frac{K_1}{|Df^k(z)|}\leqslant d(z,J_f)\leqslant \frac{K_2}{|Df^k(z)|}.$$ \end{Prop} \begin{proof}
First construct a dyadic number $R$ such that $$0<R<\min\{d(\C\sm V_3,J_f),d(V_1,\C\sm V)\}.$$ Then
for any $z,k$, satisfying to the conditions of the proposition, the open disk $U_R(f^k(z))$ does
not intersect neither $J_f$ nor a forward orbit of a critical point of $f$. Let $$\phi:
U_R(f^k(z))\ra \C$$ be the branch of $(f^k)^{-1}$ such that $\phi(f^k(z))=z$. Then by Koebe Quarter
Theorem
 the image $\phi(U_R(f^k(z)))$ contains the disk or radius
${\scr\frac{1}{4}}R|D\phi(f^k(z))|.$ Since $J_f$ is invariant under $f$ it follows that
$$d(z,J_f)\ge \frac{R}{4|Df^k(z))|}.$$ Set $K_1=R/4$.

Further, denote \be\label{PJ}PJ_f=\{f^j(c):c\in CJ_f,j\ge 0\}.\ee
Notice that $PJ_f$ is finite. Recall that we can algorithmically
construct approximate positions of all critical points of $f$ which
lie in $J_f$ with any given precision. Thus, we can approximate
$PJ_f$. We can also algorithmically construct positions of some
points in $J_f$ which lie apart from $PJ_f$ with any given
precision. For instance, using the algorithm from Proposition
\ref{root finding} we can calculate approximate position of a
repelling periodic orbit. Using the above we can construct a finite
number of pairs of simply connected dyadic sets
 $W_k\Subset U_k$ such that the following is true
 \begin{figure}\centering\epsfig{file=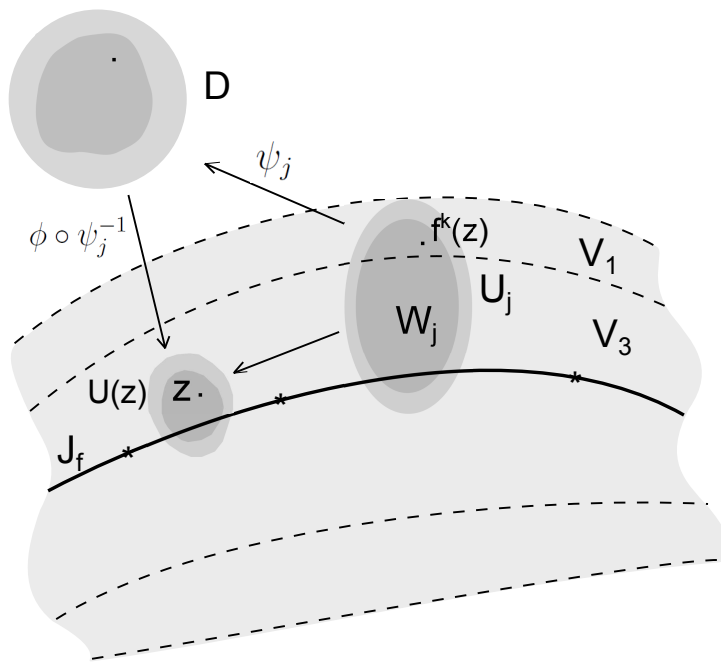,width=.70\linewidth}\caption{Illustration to Proposition
 \ref{using Koebe}.}
 \end{figure}
 \bit\item[$1)$] $W_j\cap J_f\neq \varnothing$ for any $j$;
 \item[$2)$] $\bigcup W_j\supset V_1\sm V_3$;
 \item[$3)$] $\bigcup U_j\subset V\sm PJ_f$.
 \eit
 For each $j$ fix the Riemann mapping $\psi_j:U_j\ra \mathbb{D}=\{z:|z|<1\}$.
Assume now that $z,k$ satisfy the conditions of the proposition.
Then $f^k(z)\subset W_j$ for some $j$. Let $\phi:U_j\ra U(z)$ be the
branch of $(f^k)^{-1}$ such that $\phi(f^k(z))=z$. Consider the map
$$\phi\circ\psi_j^{-1}:\mathbb{D}\ra U(z).$$ Notice that
$\psi_j(W_j)\Subset \mathbb{D}$. Applying both parts of Theorem
\ref{Koebe thm} to the map $\phi\circ\psi_j^{-1}$, we can construct
a dyadic number $r_j>0$ not depending either on $z$ or on $k$ such
that $\phi(W_j)$ is contained in the disk of radius
$r_j|D\phi(f^k(z))|$ centered at $z$. It follows that
$$d(z,J_f)\le \frac{r_j}{|Df^k(z)|}.$$ Set $K_2=\max\{r_j\}$.
\end{proof}
 As a preparatory step construct a dyadic set $W_2$ such that
$$V_3\Subset W_2\Subset V_2.$$ Compute dyadic numbers $s,\varepsilon>0$ such that \be\label{eps
s}\varepsilon<\min\{d(W_2,\C\sm V_2),d(V_3,\C\sm W_2)\},\;\; V_1\subset U_s(J_f).\ee

 Let $m=\max\{\nu(a_j)\}$. Lemma
\ref{sub and euclid} implies that for any $z\in V$ \begin{eqnarray}\label{sub eucl dist}
C^{-1}d(z,J_f)\leqslant \d_\mu(z,J_f)\leqslant mC d(z,J_f)^{1/m}. \end{eqnarray} Let $\log$ stand
for the logarithm with base 2. We will use the standard notation $[x]$ for the integer part of a
real number $x$.

 Assume that we would like to verify that a dyadic point $z$ is $2^{-n-1}$ close to $J_f$. Consider first points
 $z$ which lie outside $V_3$. Construct a dyadic set $W_3$ such that
 $$J_f\subset W_3\Subset V_3.$$ Then we can approximate the distance from a point $z\notin W_3$
 to $J_f$ by the
 distance form $z$ to $W_3$ up to a constant factor.

Now assume $z\in V_3$. Consider the following subprogram:\\$i:=1$ \\{\bf while} $i\le
[\log(mC^2s^{1/m})/\log \lambda+(n+1)/\log\lambda]+1$
{\bf do}\\
$(1)$ Compute dyadic approximations $$p_i\approx
f^i(z)=f(f^{i-1}(z))\;\;\text{and}\;\; d_i\approx
\left|Df^i(z)\right|=\left|Df^{i-1}(z)\cdot
Df(f^{i-1}(z))\right|$$ \\with precision
$\min\{2^{-n-1},\varepsilon\}$.\\ $(2)$ Check the inclusion
$p_i\in W_2$:\begin{itemize}\item[$\bullet$] if $p_i\in
W_2$, go to step $(5)$;\item[$\bullet$] if $p_i\notin W_2$,
proceed to step $(3)$;\end{itemize}
$(3)$ Check the inequality $d_i\ge K_2 2^{n+1}+1$. If true, output $0$ and exit the subprogram, otherwise\\
$(4)$ output $1$ and exit subprogram.\\
$(5)$ $i\ra i+1$\\
{\bf end while}\\
$(6)$ Output $0$ end exit.\\
{\bf end}

The subprogram runs for at most $L=[\log(mC^2s^{1/m})/\log \lambda+(n+1)/\log\lambda]+1=O(n)$
number of while-cycles each of which consist of a constant number of arithmetic operations with
precision $O(n)$ dyadic bits. Hence the running time of the subprogram can be bounded by $O(n^2\log
n\log\log n)$ using efficient multiplication. \begin{Prop}\label{subprogram} Let $f(n,z)$ be the
output of the subprogram. Then \be\left\{\ba
1,&\text{if}\;\;d(z,J_f)>2^{-n-1},\\
0,&\text{if}\;\;d(z,J_f)<K2^{-n-1},\\\text{either}\;0\;\text{or}\;1,&\text{otherwise},\ea
\right.\ee where $K=\frac{K_1}{K_2+1}$, \end{Prop}\begin{proof} Suppose first that the subprogram
runs the while-cycle $L$ times and exits at the step $(6)$. This means that $p_i\in W_2$ for
$i=1,\ldots,L$. In particular, $p_{L}\in W_2$. It follows that $f^{L}(z)\in V_1$. By (\ref{eps s})
and (\ref{sub eucl dist}) we obtain:\bes
d(z,J_f)\le C\d_\mu(z,J_f)\le C\lambda^{-L}\d_\mu(f^{L}(z),J_f)\le\\
\lambda^{-L}mC^2d(f^L(z),J_f)^{1/m}\le \lambda^{-L}mC^2s^{1/m}\le
2^{-n-1}.\ees Thus if $d(z,J_f)> 2^{-n-1}$, then the subprogram
exits at a step other than $(6)$.

Now assume that for some $i\le L$ the subprogram falls into the step
$(3)$. Then $$p_{i-1}\in W_2\;\; \text{and}\;\; p_i\notin W_2.$$ By
(\ref{eps s}), $f^i(z)\in V_1\sm V_3$. Now if $d_i\ge K_2
2^{n+1}+1$, then $|Df^i(z)|\ge K_22^{n+1}$. By Proposition
\ref{using Koebe}, $$d(z,J_f)\le 2^{-n-1}.$$ Otherwise,
$|Df^i(z)|\le K_22^{n+1}+2\le (K_2+1)2^{n+1}.$ In this case
Proposition \ref{using Koebe}
 implies $$d(z,J_f)\ge \frac{K_1}{K_2+1}2^{-n-1}.$$
\end{proof}
 Now, to distinguish the case when $d(z,J_f)<2^{-n-1}$ from the case when
$d(z,J_f)>2^{-n}$ we can partition each pixel of size $2^{-n}\times 2^{-n}$ into pixels of size
$(2^{-n}/K)\times (2^{-n}/K)$ and  run the subprogram for the center of each subpixel. This would
increase the running time at most by a constant factor.

\section{Maps without recurrent critical orbits and
parabolic periodic points.}\label{SectionNoRecNoPar} In this section
we will prove Theorem \ref{main}. Throughout this section let $f$
stand for a rational map without parabolic periodic points such that
$\Omega_f$ does not intersect the set of critical points of
$f$.\subsection{Preparatory steps and nonuniform
information.}\label{SubsecPrep} As in the case of subhyperbolic map,
without loss of generality we will assume that $\infty \notin J_f$.
For any $c\in CJ_f$ put $$N_0(c)= \max\{n:f^n(c)\in CJ_f\}+1.$$ Put
$N_0=\max\limits_{c\in CJ_f} N_0(c)$.
 Denote $$\widetilde{C}=\{f^n(c):c\in CJ_f,0\leqslant n<N_0(c)\},\;\;
 M=\overline{\{f^n(c):c\in CJ_f,n\geqslant N_0(c)\}}.$$
 By our assumptions, there are no either
recurrent critical orbits or parabolic periodic points of $f$, the
set $M$ is invariant and does not contain critical points of $f$.
Thus, the set $M$ satisfies the conditions of Ma\~{n}\'{e}'s Theorem
$(\ref{Mane thm})$. The following result is classical (see
\cite{M}).
\begin{Th}\label{Zi and He} Let $g$ be a rational map. Then the
boundary of each cycle of Siegel disks and each cycle of Herman
rings belongs to $\overline{PJ_g}$ (see (\ref{PJ})).\end{Th}
\begin{Lm} There are no either Siegel disk cycles or Herman ring
cycles in the Fatou set of $f$.\end{Lm}\begin{proof} Assume for
simplicity that there is a Siegel disk $\Delta$. By replacing $f$
with an iterate if necessary, we can assume that $f(\Delta)=\Delta$.
Then the boundary $\partial \Delta$ of the Siegel disk is forward
invariant under $f$ and belongs to $\overline{PJ_f}$. It follows
that $\partial \Delta$ does not contain any critical point of $f$.
Thus, $\partial \Delta$ satisfies the condition of Ma\~{n}\'{e}'s
Theorem $(\ref{Mane thm})$. Therefore, there exists $N$ such that
$f^N$ is expanding on a neighborhood of $\partial \Delta$. This is
impossible since $f$ is conjugated to a rotation inside $\Delta$.
The other case can be treated similarly.\end{proof}

To compute the Julia set the algorithm will use the following {\bf
non-uniform information}:\\${\bf N1.}$ $N_F(f),\;N_J(f)$ and degrees
$m_1,\ldots,m_{N_j(f)}$ of the critical points of $f$ which lie in
$J_f$;\\${\bf N2.}$ $N_0,\;N\in\mn$, dyadic numbers
$\delta,\delta'>0,q>1$ and a dyadic set $U\Supset M$ such that
$$U_{\delta/2}(M)\supset f^N(U),\;\;U\supset
U_{\delta'}(M),\;\;U_\delta(U)\cap
\widetilde{C}=\varnothing$$ and
 for any $z\in U_\delta(U)$ one has $$|Df^N(z)|>q.$$
In this section we will prove the following theorem.\begin{Th}\label{main'} Let $f$ be a rational
map such that $f$ has
 no parabolic periodic points and $\Omega_f$ does not
 contain any critical points. There exists a Turing
Machine which, given an oracle for the coefficients of the map $f$
and the non-uniform information $(N1,N2)$, computes $J_f$ in a
polynomial time.\end{Th}

 Now we prove several auxiliary lemmas.
 \begin{Lm}\label{expansion} For any $z\in U$
 and any $z_1,z_2\in U_\delta(z)$ one has $$\left|f^N(z_1)-f^N(z_2)\right|\geqslant q|z_1-z_2|.$$
\end{Lm} Notice that, in particular, the restriction of $f^N$ on $U_{\delta}(z)$ is one to one for
any $z\in U$. \begin{proof} Using Lagrange formula
\be\label{Lagrange}g(z_1)-g(z_2)=(z_1-z_2)Dg(\lambda z_1+(1-\lambda)z_2),\;\;\lambda\in[0,1],\ee
for $g(z)=f^N(z)$ we obtain $$\left|f^N(z_1)-f^N(z_2)\right|\geqslant q|z_1-z_2|.$$ \end{proof}
 For simplicity set $F=f^N$. Lemma \ref{expansion} implies that
 \be\label{EquationExpansionNearM}d(F(z),J_f)\ge qd(z,J_f)\text{ for each }z\in
 U_\delta(U).\ee\begin{Lm}\label{iter} For any $z\in U$ and $k\leqslant\min\{j\in
\mathbb{N}:F^j(z)\notin U\},k\in \mn,$ one has
\begin{eqnarray}\label{ineq1}\begin{aligned}
\frac{1}{2}\left|DF^k(z)\right|d(z,J_f)\leqslant
d(F^k(z),J_f)\leqslant \frac{9}{2}
\left|DF^k(z)\right|d(z,J_f),\\
d(F^k(z),M)\leqslant \frac{9}{2}\left|DF^k(z)\right|d(z,M).\end{aligned} \end{eqnarray}
 \end{Lm}
\begin{proof}
\begin{figure}\epsfig{file=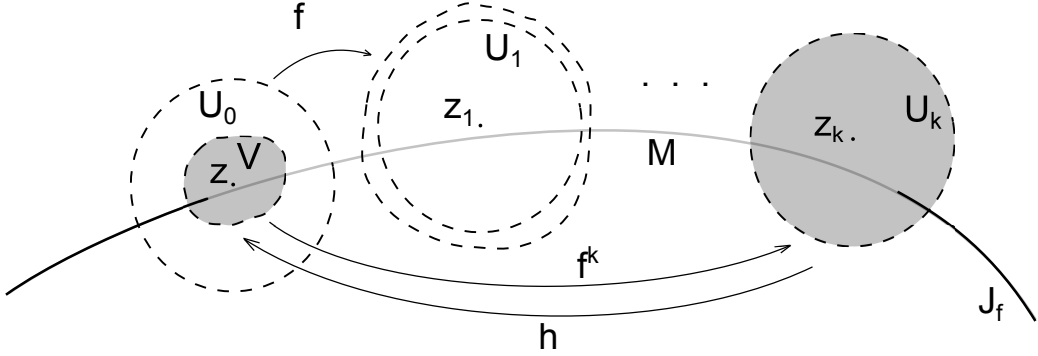,width=\linewidth}
\caption{Illustration to Lemma \ref{iter}}\end{figure} Let $z,k$
satisfy the conditions of Lemma \ref{iter}. Put
$$z_j=F^j(z), U_j=U_\delta(z_j),j=0,1,\ldots, k.$$ By Lemma
\ref{expansion}, $F$ is univalent on $U_j$ for each $j=0,1,\ldots,
k-1$ and $$F(U_j)\supset U_{j+1}.$$ It follows that there exists
$V\subset U_0$, such that
$$F^k(V)=U_k$$ and $F^k$ is univalent on $V$. Lemma
\ref{expansion} implies that $$V\subset
U_{\delta/q}(z_0).$$ Let $h:U_k\rightarrow V$ such that
$h\circ F^k$ is identical on $V$. Put $$r=d(z_k,J_f).$$
Then $r\leqslant\frac{\delta}{2}$.  By Koebe Distortion
Theorem (see (\ref{Koebe})), one has:\begin{eqnarray}
 U_{\frac{2}{9}r|h'(z_k)|}(z)
 \subset h\left(U_{r}(z_k)\right)
\subset U_{2r|h'(z_k)|}(z)  .\end{eqnarray} Since $J_f$ is invariant
under $f$, it follows that
\begin{eqnarray}\label{ineq2}\frac{2}{9}r|h'(z_k)|\leqslant
d(z,J_f)\leqslant 2r|h'(z_k)|.\end{eqnarray} Since
$Dh(z_k)=\frac{1}{DF^k(z)}$, $(\ref{ineq2})$ is equivalent to the
first part of $(\ref{ineq1})$. Similarly, one can prove the second
part of $(\ref{ineq1})$. Observe that we can not prove a lower bound
for $d(F^k(z),M)$ in this way since $M$ is not necessarily invariant
under $F^{-1}$.\end{proof} Till the end of this section in each of
the preceding statements we assume that the Turing Machine has an
access to the non-uniform information $(N1,N2)$ and an oracle for
the coefficients of the map $f$.\begin{Lm}\label{crit dist} There
exists a TM which computes dyadic numbers ${\varepsilon_1}>0,K>0$,
 such that for any critical point $c$ of $f$ which lie in $J_f$ one has
  $$d(f(z),J_f)\geqslant K d(z, J_f)|z-c|^{m-1}$$ for any $z\in U_{\varepsilon_1}(c)$,
  where $m$ is the
  degree of $f$ at $c$.
\end{Lm} \begin{proof} By Proposition \ref{root finding} we
can find approximate positions of all critical points of $f$. Notice
that the orbit of every critical point of $f$ which lie in the Fatou
set converges to an attracting cycle. Using the ideas from the proof
of Proposition \ref{attr orbits finding} we can distinguish the
critical points which belong to the Fatou set from the critical
points which belong to the Julia set. Let $c\in J_f$ be a critical
point of degree $m$. Then we can approximate the coefficient $a_m$
at $(z-c)^m$ in the Taylor expansion of $f$ at $c$ and compute
$\gamma>0$, such that
$$\left|\frac{f(z)-f(c)}{a_m(z-c)^m}-1\right|<1$$ for any
$z\in U_\gamma(c)$. Let $\omega$ be any $m$-th root of
$a_m$. Then there exists a unique holomorphic map $\psi$
from $U_\gamma(c)$ to a neighborhood of the origin, such
that
$$f(z)=\psi(z)^m+f(c)\;\;\text{on}\;\;U_\gamma(c)\;\;\text{and}\;\;\psi'(c)=\omega.$$
Make $\gamma$ smaller if necessary in order to have
$$|\psi'(z)|\geqslant |\omega|/2\text{ for any }z\in
U_\gamma(c).$$ The image $f(U_\gamma(c))$ contains a disk
$U_\alpha(f(c))$, where $\alpha>0$ can be algorithmically
constructed. Let $0<{\varepsilon_1}<\gamma$, such that
$$f(U_{\varepsilon_1}(c))\subset U_{\alpha/2}(f(c)).$$ Take
$z\in U_{\varepsilon_1}(c)$. Observe that
$$d(f(z),J_f)\leqslant d(f(z),f(c))\leqslant \alpha/2.$$
Let $\zeta_0\in J_f$, such that
$$d(f(z),J_f)=|f(z)-\zeta_0|.$$ Then $\zeta_0\in
U_\alpha(f(c))$. There exists $z_0\in J_f\cap U_\gamma(c)$,
such that $f(z_0)=\zeta_0$. One has
$$f(z)-f(z_0)=\psi(z)^m-\psi(z_0)^m=\prod\limits_{j=0}^{m-1}\left(\psi(z)-e^{2\pi
i j/m}\psi(z_0)\right).$$ Let $0\leqslant j_0\leqslant m-1$
be the number for which $e^{2\pi i j_0/m}\psi(z_0)$ takes
the closest value to $\psi(z)$. From simple geometric
observations it follows that $$\left|\psi(z)-e^{2\pi i
j/m}\psi(z_0)\right|\geqslant
|\psi(z)|\sin(\pi/m)\geqslant\frac{2|\psi(z)|}{m}\geqslant
\frac{|\omega|}{m}|z-c|$$ for any $j\neq j_0$. Since
$\psi^{-1}\left(e^{2\pi i j_0/m}\psi(z_0)\right)\in J_f$,
it follows that \begin{eqnarray*}d(f(z),J_f)\geqslant
\left(\frac{|\omega||z-c|}{m}\right)^{m-1}\left|\psi(z)-e^{2\pi
i j_0/m}\psi(z_0)\right|\geqslant K
 |z-c|^{m-1}d(z,J_f),\end{eqnarray*} where $K=m^{1-m}|\omega|^m/2$.
\end{proof} Using the definition of the dyadic set $U$ and
Proposition \ref{root finding} one can prove the following
statement.\begin{Lm}\label{eps2} There exists a TM which
computes a dyadic number $\varepsilon_2>0$ such
that\\
\be\label{eps2 formula}\overline{U_{\varepsilon_2}\left(\widetilde{C}\right)} \cap
\overline{f^N(U)}=\varnothing\text{ and }
 f^{N_0}\left(\overline{U_{\varepsilon_2}\left(\widetilde{C}\right)} \right)\subset
U.\ee\end{Lm}
\begin{Prop}\label{critical}
  There exists a TM
  which computes dyadic number
  $0<\gamma<\min\{\varepsilon_1,\varepsilon_2\}$ such that the following is true. For any
  $z\in U_\gamma\left(\widetilde{C}\right)$ and
$k=\min\{j\in \mathbb{N}:f^{Nj+N_0}(z)\notin U\}$
 one has $$d(f^{kN+N_0}(z),J_f)\geqslant qd(z,J_f).$$\end{Prop} \begin{proof}
 First put $\gamma=\min\{\varepsilon_1,\varepsilon_2\}$. Let
$z\in U_{\gamma}\left(\widetilde{C}\right)\setminus J_f$. Set
$$w=f^{N_0}(z), k=\min\{j\in \mathbb{N}:F^j(w)\notin U\},$$  where $F=f^N$.
Using Lemma \ref{iter}, we obtain: \begin{eqnarray}\label{ineq3}
 \frac{d(F^k(w),J_f)}{d(w,J_f)}\geqslant
 \frac{d(F^k(w),M)}{9d(w,M)}.
\end{eqnarray} Let $c\in \widetilde{C}$. Put $h=f^{N_0}$.
   Let $m_c$ be the degree of $h$ at $c\in \widetilde{C}$.
   Then we can algorithmically construct a dyadic number $L_c>0$, such that
$$|h(z)-h(c)|\leqslant L_c|z-c|^{m_c}$$ for any $z\in U_\gamma(c)$. Then
\be\label{ineq5}d(h(z),M)\le d(h(z),h(c))\leqslant
L_c|z-c|^{m_c}.\ee Lemma \ref{crit dist} implies that we can
algorithmically construct dyadic numbers $\alpha_c>0, K_c>0$ such
that
   \begin{eqnarray}\label{ineq6}
 d(h(z),J_f)\geqslant K_c d(z,J_f)|z-c|^{m_c-1}
\end{eqnarray} for any $z\in U_{\alpha_c}(c)$.
Combining $(\ref{ineq5})$ with $(\ref{ineq6})$, we get \begin{eqnarray}\label{ineq4}
 \frac{d(h(z),J_f)}{d(z,J_f)}\geqslant K_cL_c^{-1}\frac{d(h(z),M)}{|z-c|}
\end{eqnarray} for any $z\in U_{\min\{\alpha_c,\gamma\}}(c)\setminus J_f$.

Take $\gamma$ such that $\gamma\leqslant \min\left\{ \frac{\delta'
K_cL_c^{-1}}{9q},\alpha_c\right\}$ for any $c\in \widetilde{C}$.
Recall that in $(\ref{ineq3})$ $F$ stands for $f^{N}$ and $w$ stands
for $f^{N_0}(z)$,
 in $(\ref{ineq4})$ $h$ stands for $f^{N_0}$.
Now (\ref{ineq4}) implies that
\begin{eqnarray}\label{ineq7}
 \frac{d(w,J_f)}{d(z,J_f)}\geqslant 9q\frac{d(w,M)}{\delta'}
\end{eqnarray} for any $z\in U_\gamma\left(\widetilde{C}\right)$.
 Combining $(\ref{ineq3})$ and
$(\ref{ineq7})$, taking into account that $U\supset
U_{\delta'}(M)$, we obtain the statement of Lemma
\ref{critical}. \end{proof} A direct analog of the
Proposition \ref{set U} holds for the maps $f$, which
satisfy the conditions of Theorem \ref{main}.
\begin{Prop}\label{set U nonrec} Let $f$ be a rational map
such that $f$ does not have any parabolic periodic points
and $\Omega_f$ does not contain any critical points of $f$.
There exists a TM which given an oracle for the
coefficients of the map $f$ and the non-uniform information
outputs a planar domain $V\in \mathcal{C}$ such
that:\begin{itemize} \item[(1)] $V\Subset f(V)$, \item[(2)]
$f^2(V)\cap PF_f=\varnothing$, \item[(3)] $J_f\Subset V$.
\end{itemize} \end{Prop}

  Denote $$\Theta=
CJ_f\cup\bigcup\limits_{c\in
CJ_f}\overline{\mathcal{O}(c)}=\widetilde{C}\cup M,\;\;
W=V\setminus\Theta.$$ If $CJ_f=\varnothing$ then $f$ is
hyperbolic. Since for hyperbolic maps poly-time
computability of the Julia set is well known (see
\cite{Bra04} and \cite{Ret05}), we will assume that
$CJ_f\neq \varnothing$. Then $f(\Theta)$ is strictly
smaller than $\Theta$. It follows that $f(W)$ is strictly
larger than $W$. Let $\|\cdot\|_ W$ stands for the
hyperbolic norm associated with $ W$.
 Then $$\|Df(z)\|_W>1\text{ for every }z\in  W.$$ Let $d_ W$ be the
metric on $ W$ induced by $\|\cdot\|_ W$. It follows that
\be\label{hyper expansion}
 d_ W(f(z),J_f)\geqslant d_ W(z,J_f)\text{ for every }z\in f^{-1}(W)\cap W.\ee  Notice that
$$\inf\limits_{z \in J_f\setminus
\left(U_\gamma\left(\widetilde{C}\right) \right)
}|f'(z)|>0.$$
 One can construct a dyadic subset $V_1$, such that
$$f^{-1}(V)\Subset V_1\Subset V.$$ Denote $W_1=V_1\setminus
\Theta$. \begin{Lm}\label{LemmartR} One can algorithmically
construct dyadic numbers $r\in(0,1], t>1,R>0,\epsilon>0$,
 such that
\begin{itemize} \item[(1)] $d_ W(f(z),J_f)\geqslant t\,d_
W(z,J_f)$ for any $z\in W\setminus \left(U_\gamma\left(
\widetilde{C}\right)\cup U\right) $; \item[(2)]
$d(f(z),J_f)\geqslant r d(z,J_f)$ for any $z\in W\setminus
U_\gamma\left( \widetilde{C}\right)$; \item[(3)]
$D_R(0)\supset W\supset W_1\supset U_\epsilon(J_f)$.
\end{itemize}\end{Lm} Notice that the Euclidian metric $d$
and the hyperbolic metric $d_ W$ are equivalent on any
compact subset of $ W$. The following lemma can be proven
similarly to Lemma \ref{sub and euclid}.
  \begin{Lm}\label{LemmaMetrics} One can algorithmically construct a dyadic constant $C>0$ such that
  \begin{eqnarray}\label{equi norms}C^{-1}d(z,J_f)\leqslant d_W(z,J_f)\leqslant C d(z,J_f)\end{eqnarray}
   for any $z\in W_1\setminus
\left(U_\gamma\left( \widetilde{C}\right)\cup U \right)
$.\end{Lm} \begin{Lm}\label{separated} There exists a TM
which computes dyadic numbers $S,K>0$ such that for any
$k\in \mathbb{N}$ and $z\in \mathbb{C}$ one has
$d\left(f^k(z),J_f\right)\geqslant Sd(z,J_f)^K$. \end{Lm}
 \begin{proof}
Since $f$ does not have parabolic periodic orbits one has
$$\inf_{n\in \mathbb{N}}d\left(f^n\left(\mathbb{C}\setminus
V_1\right),J_f\right)>0.$$ Therefore there exist
$S<1,K_0>1$ such that for any $z\notin V_1$ and $k\in
\mathbb{N}$ $$d\left(f^k(z),J_f\right)> Sd(z,J_f)^{K_0}.$$
Make $S$ smaller if necessary in order to have $$ d(\mc\sm
V_1,J_f)> S\sup\limits_{z\in V_1}d(z,J_f)^{K_0}.$$
 Let $z_0\in  V_1\setminus J_f$.
 Assume there exists $l\in \mathbb{N}$ such that
$$d(f^l(z_0),J_f)\leqslant Sd(z_0,J_f)^{K_0}.$$ Let $l$ be
the minimal such number. Then $$f^j(z_0)\in V_1\text{
 for }j=0,1,\ldots, l.$$ It follows from Lemma \ref{crit dist} that one can compute dyadic
 numbers
 $1\ge \eta>0$, $m\ge 1$ such that
\begin{eqnarray}\label{eta m}d\left(f^j(z),J_f\right)\geqslant \eta
d\left(z,J_f\right)^m\end{eqnarray} for any $z\in U_\gamma\left(\widetilde{C} \right)$ and
$j=1,2,\ldots, N_0+N-1$.
 Put $l_0=0$.
 Denote inductively
 $l_i$ and $z_i=f^{l_i}(z_0)$ as follows:
\begin{eqnarray}\label{EquationLiZi}
 l_{i+1}=\left\{ \begin{array}{ll}l_i+1,&\text{if}\;\;z_i\notin U_\gamma\left(\widetilde{C} \right)\cup
 U,\\ l_i+N,&\text{if} \;\;z_i\in U\;\text{and}\;l_i+N\leqslant l
,\\l_i+N_0+kN, &\text{if} \;\;z_i\in U_\gamma\left(\widetilde{C} \right)\;\text{and}\;l_i+N_0+kN\leqslant l,\\
l,&\text{otherwise}. \end{array} \right. \end{eqnarray}
while $l_i<l$, where $$k=\min\{j\in
\mathbb{N}:f^{Nj+N_0}(z_i)\notin U\}.$$ Now by (\ref{hyper
expansion}), Lemma \ref{LemmartR} and Proposition
\ref{critical} the following is true:
 \begin{itemize}
\item[1)] $d_ W(z_{i+1},J_f)\geqslant d_ W(z_i,J_f)$ for any $i$; \item[2)] if $z_i\notin
U_\gamma\left(\widetilde{C}\right)\cup\, U$, then
 $$d_ W(z_{i+1},J_f)\geqslant t\,d_ W(z_i,J_f)\text{ and }d(z_{i+1},J_f)\geqslant r d(z_i,J_f);$$
\item[3)] if $z_i\in U$ and $l_i+N\le l$, then $$d(z_{i+1},J_f)\geqslant qd(z_i,J_f);$$ \item[4)]
if $z_i\in U$ and $l_i+N>l$, then $$d(z_{i+1},J_f)\geqslant r^{N-1}d(z_i,J_f);$$
 \item[5)] if $z_i\in U_\gamma\left(\widetilde{C}\right)$ and $l_i+N_0+kN\leqslant l$, then
$$d(z_{i+1},J_f)\geqslant \,d(z_i,J_f);$$
 \item[6)] if $z_i\in
U_\gamma\left(\widetilde{C}\right)$ and $l_i+N_0+kN>l$,
then $$d\left(z_{i+1},J_f\right)\geqslant \eta r^{N-1}
d\left(z_i,J_f\right)^m.$$ \end{itemize} The items $1)-5)$
are direct corollaries of the choice of the dyadic
constants, formula (\ref{hyper expansion}), Lemma
\ref{LemmartR} and Proposition \ref{critical}. Let us proof
$6)$. Let $$z_i\in U_\gamma\left(\widetilde{C}\right)\text{
and }l_i+N_0+kN>l.$$ If $l_i+N_0>l$, then $6)$ follows
directly from $(\ref{eta m})$. Otherwise denote $n$ the
remainder of $l-l_i-N_0$ modulo $N$. Then by $(\ref{eta
m})$ and (\ref{EquationExpansionNearM})
$$d\left(z_{i+1},J_f\right)\geqslant
r^{N-1}d\left(f^{l-n}(z_i),J_f\right)\geqslant
r^{N-1}d\left(f^{l_i+N_0}(z_i),J_f\right)\ge r^{N-1}\eta
d\left(z_i,J_f\right)^m.$$

Let $s$ be the number of $i$-th, for which $$z_i \notin
U_\gamma\left(\widetilde{C}\right)\cup U.$$ If $s\neq 0$
then let $I_1$
be the minimal and $I_2$ be the maximal such $i$. \\
{\bf case 1: $s=0$.} Then one of the following two possibilities holds.\\
{\bf a)} $z_0\in U_\gamma\left(\widetilde{C}\right)$. Then $l=l_1$. By $6)$,
$$d\left(f^l(z_0),J_f\right)\geqslant \eta r^{N-1} d\left(z_0,J_f\right)^m.$$
 {\bf b)} $z_0\in U$.
Then by $3)$ and $4)$,
$$d\left(f^l(z_0),J_f\right)\geqslant r^{N-1}
d\left(z_0,J_f\right).$$
\\ {\bf case 2: $s\neq 0$.} There are two possibilities.\\
{\bf a)} $l_{I_1}=l$. Then $I_1=I_2$. As in the case 1, $$d\left(f^l(z_0),J_f\right)\geqslant \eta
r^{N-1}d\left(z_0,J_f\right)^m.$$
 {\bf b)} $l_{I_1}<l$. Then, by $3)$ and $5)$, $d(z_{I_1},J_f)\geqslant d(z_0,J_f)$.
 By $(\ref{equi norms})$,
\begin{eqnarray}\label{ineq7}
 d(z_{I_2},J_f)\geqslant C^{-1}d_ W(z_{I_2},J_f)\geqslant C^{-1}
d_ W(z_{I_1},J_f)\geqslant C^{-2} d(z_0,J_f). \end{eqnarray} Among the numbers
$I_2+1,I_2+2,\ldots,l$ there is at most one number $j$ such that $$z_j\in
U_\gamma\left(\widetilde{C}\right)$$ and at most one number $j$ such that $$z_j\in U,\text{ but
}l_j+N>l.$$ It follows that \begin{eqnarray}\label{ineq8}d(f^l(z_0),J_f)\geqslant r^{m+N}\eta
d\left(z_{I_2},J_f\right)^m\geqslant Dd(z_0,J_f)^m,\end{eqnarray} where $D=r^{m+N}\eta C^{-2m}.$

Thus, in both case 1 and case 2 $$d(f^l(z_0),J_f)\geqslant
Dd(z_0,J_f)^m.$$ It follows that Lemma \ref{separated}
holds for $S=D$ and $K=\max\{K_0,m\}.$
\end{proof} \begin{Th}\label{PropositionPolytimeLeaving}
 There exists a TM which computes the coefficients of a polynomial $p(n)$
 such that for any $z$ with $d(z,J_f)>2^{-n}$ one has
$$f^{p(n)}(z)\notin V.$$ \end{Th} \begin{proof}
 Let $z_0\in V_1$ and $d(z_0,J_f)>2^{-n}$. Let $l$ be the maximal number, such that
$$f^j(z_0)\in V_1\text{ for }j=0,1,\ldots,l.$$ Construct
$l_i,z_i$ as in Lemma \ref{separated}. Then properties
$1)-6)$ hold. Let $s, I$ be as in Lemma \ref{separated} and
$i_1,\ldots i_s$
 be the sequence of indexes for which $$z_i\notin U_\gamma\left(\widetilde{C} \right)\cup
U.$$ Then, by Lemma \ref{LemmaMetrics} and properties $1)$
and $2)$ from Lemma \ref{separated}, $$R\ge
d(f^l(z),J_f)\ge C^{-2}t^sd(z,J_f)\ge C^{-2}t^s 2^{-n}.$$
It follows that $$s\leqslant
n\log\limits_t2+2\log\limits_tC+\log\limits_tR.$$ Let $i$
be an index such that $z_i\in U_\gamma\left(\widetilde{C}
\right)$ and $l_i+N_0\leqslant l$. Let $$k_i=\min\{j\in
\mathbb{N}:f^{Nj+N_0}(z_i)\notin U\}.$$ It follows from
Lemma \ref{separated} that $$d(z_i,J_f)\ge 2^{-Kn},\text{
and }\delta\ge d(f^{k_i}(z_i),J_f)\ge q^{k_i}d(z_i,J_f).$$
Therefore, $$k_i\leqslant K_2n$$ for some number $K_2>0$
which can be obtained algorithmically. Now it is easy to
see that among any $N_0+K_2Nn$ consecutive integer numbers
between $0$ and $l$ there exist $i$ such that $$z_i\notin
U_\gamma\left(\widetilde{C} \right)\cup U.$$
 It follows that $$l\leqslant
 (n\log\limits_t2+2\log\limits_tC+\log_tR)(N_0+K_2Nn).$$
 Since $f^{-1}(V)\Subset V_1\Subset V\Subset f(V)$,
 this finishes the proof.\end{proof}
 \subsection{The algorithm.}\label{SubsecAlgNonrec} Similarly to Proposition \ref{using Koebe} one can prove the
following result. \begin{Prop}\label{using Koebe nonrec} There is an algorithm computing two dyadic
constants $K_1,K_2>0$
 such that for any
$z\in V$ and any $k\in \mathbb{N}$ if $f^k(z)\in f(V)\setminus
f^{-1}(V)$ then one has $$ \frac{K_1}{|Df^k(z)|}\leqslant
d(z,J_f)\leqslant \frac{K_2}{|Df^k(z)|}.$$ \end{Prop} Now we are
ready to describe the algorithm. The steps of the algorithm are
analogous to the steps of the corresponding algorithm for a
subhyperbolic map. Compute dyadic numbers $s,\varepsilon>0$ such
that \be\label{eps s}\varepsilon<\min\{d(V_1,\C\sm f(V)),d(V,\C\sm
V_1)\},\;\; f(V)\subset U_s(J_f).\ee

 Assume that we would like to verify that a dyadic point $z$ is $2^{-n-1}$ close to $J_f$.
 Construct a dyadic set $U_2$ such that
 $$J_f\subset U_2\Subset f^{-1}(V).$$ Then we can approximate the distance from a point $z\notin U_2$
 to $J_f$ by the
 distance form $z$ to $U_2$ up to a constant factor.

Now assume that $z\in U_2$. Let $p(n)$ be the polynomial
from Proposition \ref{PropositionPolytimeLeaving}.
Similarly to the case of a subhyperbolic map, consider the
following subprogram:\\$i:=1$ \\{\bf while} $i\le p(n+1)$
{\bf do}\\ $(1)$ Compute dyadic approximations $$p_i\approx
f^i(z)=f(f^{i-1}(z))\;\;\text{and}\;\;d_i\approx
\left|Df^i(z)\right|=\left|Df^{i-1}(z)\cdot
Df(f^{i-1}(z))\right|$$ \\with precision
$\min\{2^{-n-1},\varepsilon\}$.\\ $(2)$ Check the inclusion
$p_i\in V_1$:\begin{itemize}\item[$\bullet$] if $p_i\in
V_1$, go to step $(5)$;\item[$\bullet$] if $p_i\notin V_1$,
proceed to step $(3)$;\end{itemize}
$(3)$ Check the inequality $d_i\ge K_2 2^{n+1}+1$. If true, output $0$ and exit the subprogram, otherwise\\
$(4)$ output $1$ and exit subprogram.\\
$(5)$ $i\ra i+1$\\
{\bf end while}\\
$(6)$ Output $0$ end exit.\\
{\bf end}

The subprogram runs for at most $p(n+1)$ number of while-cycles each
of which consist of a constant number of arithmetic operations with
precision $O(n)$ dyadic bits. The following proposition is proved in
the same way as Proposition \ref{subprogram}.
\begin{Prop} Let $f(n,z)$ be the output of the subprogram.
Then \be f(n,z)=\left\{\ba
1,&\text{if}\;\;d(z,J_f)>2^{-n-1},\\
0,&\text{if}\;\;d(z,J_f)<K2^{-n-1},\\\text{either}\;0\;\text{or}\;1,&\text{otherwise},\ea
\right.\ee where $K=\frac{K_1}{K_2+1}$, \end{Prop}
\section{Maps with parabolic periodic points.}
  In this section we will sketch the proof of Theorem
  \ref{main1}. Let $f$ be a rational map such that
  $\Omega_f$ does not contain either critical points or
  parabolic periodic points. We will assume that the
  periods and multipliers of the parabolic periodic points
  are given as a part of the non-uniform information. Replacing $f$ with some
iteration of $f$ if necessary, we may assume that the
multiplier and the period of each of the parabolic periodic
points is equal to $1$. Then at each parabolic
  periodic point (in fact, a fixed point) $p$ the map $f$ can be written in the form
$$f(z)=z+(z-p)^{n_p+1}+O(z^{n_p+2}),$$ where $n_p\in \mn$.
We will assume that the numbers $n_p$ are also given as a
part of the non-uniform information.  Observe that some
parabolic fixed points of $f$
 may belong to the postcritical set of $f$.

In this section we denote by $CJ_f$ the set of critical
points $c\in CJ_f$ such that the forward orbit
$\mathcal{O}(z)$ does not contain any parabolic periodic
points and by $CJ_f^P$ the set of critical points $c\in
J_f$ such that an iteration of $c$ hits a parabolic fixed
point. Let $N_0(c),N_0,M,\widetilde{C}$ be the same as in
section \ref{SectionNoRecNoPar}. For numbers
$\alpha,\beta>0$, a direction $\nu\in[0,2\pi]$ and a point
$w\in \mc$ introduce a sector
\be\label{EquationSector}V_w^\nu(\beta,\alpha)=\{z\in
U_\alpha(w):\arg(z-w)\in(\nu-\beta,\nu+\beta)\}.\ee

 Let $p_1,\ldots,p_s$ be
all parabolic fixed points of $f$. The algorithm will use the same
{\bf non-uniform information} $(N1,N2)$ which was described in
subsection \ref{SubsecPrep} and, in addition,
\\${\bf N3.}$ approximate positions $a_j$ of $p_j$ with a dyadic
precision $\alpha>0$ and a number $\beta>0$
such that for each $j$:\\
$1)$ $p_j$ is a unique fixed point of $f$ in
$U_{2\alpha}(a_j)$ (and thus $p_j$ can be approximated
efficiently using Newton method);\\
$2)$ for each critical point $c$ which lie in $J_f$ one has
$$U_{2\alpha}(a_j)\setminus \{p_j\}\cap
\mathcal{O}(c)=\varnothing;$$\\
\begin{figure}\centering\epsfig{file=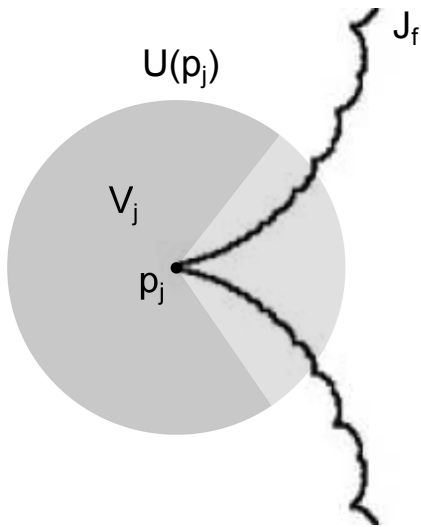,width=.40\linewidth}\caption{Illustration to
non-uniform information $N3$.} \end{figure}
$3)$ for each attracting
direction $\nu\in[0,2\pi]$ at $p_j$ the sector
$$V_j^\nu=V_{p_j}^\nu(\beta,2\alpha)$$ belongs to an attracting
Fatou petal at $p_j$, $$\overline{V_j^\nu}\cap J_f=\{p_j\}$$ and the
set
$$U_{2\alpha}(p_j)\setminus ((\cup V_j^\nu) \cup \{p_j\})$$
belongs to the union of the repelling petals at $p_j$.
 Observe that for a point
$z$ in the sector $V_j^\nu$ the distance $d(z,J_f)$ up to a constant
factor can be approximated by $|z-p_j|$. We will use the following
result from \cite{Bra06} (see Lemma
$8$).\begin{Lm}\label{LemmaLongIter} Let
$g(z)=z+c_{n+1}z^{n+1}+c_{n+2}z^{n+2}+\ldots$ be given as a power
series with radius of convergence $R>0$. There is an algorithm which
given a point $z$ with $|z|<1/m<R$ computes $l=[m^n/C]$-th iteration
of $z$ and the derivative $Dg^l(z)$ with precision $2^{-s}$ in time
polynomial in $s$ and $\log m$. Here $C$ is some dyadic constant
which can be algorithmically constructed. The algorithm uses an
oracle for the coefficients $c_n$.\end{Lm} Lemma \ref{LemmaLongIter}
implies that we can compute "long" iterations of points close to
parabolic periodic points efficiently.

By the assumptions on the map $f$ one can algorithmically
construct a dyadic number $\varepsilon>0$ such that the
following is true. Let a point $z$ be $\varepsilon$-close
to the Julia set. Assume that some iteration $f^k(z)$
belong to $U_\varepsilon(p_j)$. Consider the neighborhood
$V=U_\varepsilon(f^k(z))$. Let
$$V_0=U(z),V_1,\ldots,V_k=V$$ be the pullbacks of $V$ under
$f$ along the orbit of $z$, where $U(z)$ is a neighborhood
of $z$. Then the number of $n$ for which $V_n$ contains a
critical point $c\in J_f$ is bounded from above by a
constant independent from $k$ and $z$. In a similar fashion
as Proposition \ref{using Koebe} using ideas of Lemma
\ref{crit dist} we can prove the following statement.
\begin{Prop}\label{using Koebe par} We can algorithmically
construct dyadic numbers $K_1,K_2,\varepsilon>0$ such that
for each $z\in U_\varepsilon(J_f)$ if $k$ satisfies the
following $$f^n(z)\in U_\varepsilon(J_f)\;\text{ for
}\;n=1,\ldots, k,\;\text{ and }\;f^k(z)\in
U_\varepsilon(p_j)$$ for some parabolic fixed point $p_j$
then $$\frac{K_1d(f^k(z),J_f)}{|Df^k(z)|}\leqslant
d(z,J_f)\leqslant
\frac{K_2d(f^k(z),J_f)}{|Df^k(z)|}.$$\end{Prop} From the
description of the dynamics near a parabolic point using
Koebe Theorem \ref{Koebe thm} we can obtain the following
result. \begin{Prop}\label{PropParEsc} One can
algorithmically construct numbers $K_3,L>0$ such that for
each $z\in U_\varepsilon(p_j)\setminus (\cup V_j^\nu)$ and
each $n\in \mn$ if $$f^i(z)\in U_\varepsilon(p_j)\text{ for
all }i=0,1,\ldots,l=2^{[Ln]}$$ then one has
$|f^l(z)-p_j|\ge 2^n|z-p_j|$ and
$$K_3^{-1}\frac{|f^l(z)-p_j|}{|z-p_j|}\le\frac{d(f^l(z),J_f)}{d(z,J_f)}\le
K_3\frac{|f^l(z)-p_j|}{|z-p_j|}.$$ \end{Prop}Now we briefly
explain how to adopt the algorithm from Paragraph
\ref{SubsecAlgNonrec} to prove Theorem \ref{main1}. Assume
that we want to verify if a point $z$ is $2^{-n}$-close to
$J_f$. Construct a sequence $z_i=f^{l_i}(z)$ in a way
analogues to the construction from Lemma \ref{separated}
(see (\ref{EquationLiZi})). We will define $l_{i+1}$ in a
different way from (\ref{EquationLiZi}) only if $z_i\in
U_\varepsilon(p_j)$ for some parabolic fixed point $p_j$.
In this case we do the following.\\
$1)$ If $z_i\in V_j^\nu$ for some attracting direction
$\nu$ then we stop. We can find the distance from $z_i$ to
$J_f$ up to a constant factor. Using Proposition \ref{using
Koebe par} we can estimate the distance from $z$ to $J_f$
up to a constant factor.\\
$2)$ If $z_i\in U_\varepsilon(p_j)\setminus (\cup V_j^\nu))$ then
consider the points $f^{k_r}(z_i)$, where $k_r=2^{[Lr]}$. Observe
that by Lemma \ref{LemmaLongIter} we can find a $2^{-n}$
approximation of $f^{k_r}(z_i)$ in time polynomial in $r$ and $n$.
Let $r$ be the minimal nonnegative integer number such that
$$|f^{k_r}(z_i)-p_j|\ge \varepsilon.$$ Set
$z_{i+1}=f^{k_r}(z_i)$.

Observe that a direct analog of Theorem
\ref{PropositionPolytimeLeaving} is not true in a presence of
parabolic periodic points, even if we assume that $z$ does not
belong to an attracting basin of a parabolic periodic point. For a
point $2^{-n}$ close to a parabolic fixed point in a repelling petal
it takes exponential time in $n$ to escape an
$\varepsilon$-neighborhood of the parabolic fixed point. However,
using Lemma \ref{LemmaLongIter} and Proposition \ref{PropParEsc} we
can prove the following. \begin{Prop}\label{ProPolyTimePar} There is
an algorithm computing coefficients of a polynomial $p(n)$ such that
if $$d(z,J_f)>2^{-n}\;\text{ and }\;z_i\notin \cup V_j^\nu\;\text{
for each }\;i=0,1,\ldots,p(n),$$ then $d(z_{p(n)},J_f)>\varepsilon$.
Moreover, we can compute $2^{-n-2}$-approximations of
$z_i=f^{l_i}(z)$ and $Df^{l_i}(z)$ for $i=0,\ldots,p(n)$ in time
polynomial in $n$.\end{Prop} Now we briefly describe the algorithm
computing the Julia set $J_f$ in a polynomial time in $n$. Let $z\in
U_\varepsilon(J_f)$. Assume that we want to verify  that $z$ is
$2^{-n}$ close to $J_f$. Without loss of generality we may assume
that $2^{1-n}<\varepsilon$. Compute approximate values $a_i$ of
$z_i$ and $d_i$ of $Df^{l_i}(z_i)$ with precision $2^{-n-2}$,
$i=1,\ldots,p(n)$. \\ $1)$ If $z_i\in \cup V_j^\nu$ for some $i$
such that $$d(z_j,J_f)\le \varepsilon/2\;\;\text{for}\;\;j=1,\ldots,
i$$ then we can find approximate distance from $z_i$ to $J_f$. By
Proposition \ref{using Koebe par}, we can find $d(z,J_f)$ up to some
constant.\\$2)$ If $d(z_i,J_f)\ge \varepsilon/2$ for some $i$, then
we can find the distance $d(z_i,J_f)$ up to a constant factor. Using
Koebe Theorem $(\ref{Koebe thm})$ we can find the distance
from $z$ to $J_f$ up to a constant factor.\\
$3)$ If neither $1)$ nor $2)$ holds then by Proposition
\ref{ProPolyTimePar} $$d(z,J_f)\le 2^{-n}.$$

In conclusion, let us mention an alternative numerical method to
calculate iterations of points close to parabolic periodic orbits.
This method has been implemented in recent literature. To illustrate
the method we consider a map with a simple parabolic point at the
origin:$$h(z)=z+z^2+ \sum\limits_{k=3}^\infty a_kz^k.$$ Let
$j(z)=-1/z$  and $F(z)$ be the germ at infinity given by
 $$F(z)=j\circ f\circ j(z)=-1/f(-1/z)=z+1+a(z),\;\;\text{where}\;\;a(z)=O(z^{-1}).$$ Denote by $T$ the unit shift:
 $T(z)=z+1$. In \cite{Eca1} and \cite{Eca2}
Ecalle has shown the following
\begin{Th}\label{ThEcalle}
 The equation $\Psi\circ F(z)=T\circ\Psi(z)$ has a unique formal solution in terms of the series (generally,
divergent)
\begin{eqnarray}\label{EqnPsi}\widetilde{\Psi}(z)=\rho \log z+\sum\limits_{k=1}^\infty c_kz^{-k}.\end{eqnarray}
The series $\widetilde{\Psi}$ gives an asymptotic expansion for an
attracting and a repelling Fatou coordinates of the map $F$,
$\Psi_a$ and $\Psi_r$ correspondingly.
\end{Th}
For survey on divergent series and asymptotic expansions we refer
the reader to \cite{Ram}. Theorem \ref{ThEcalle} means that the
Fatou coordinates $\Psi_a$ and $\Psi_r$ near infinity can be
approximated by finite sums of the series (\ref{EqnPsi}). To
approximate $n$-th iteration of the map $f$ near the origin one can
use the following formula:
$$f^n(z)=j\circ\Psi^{-1}\circ T^n\circ \Psi\circ j,$$ where $\Psi$ stands for either $\Psi_a$ or $\Psi_r$
depending on whether $z$ belongs to an attracting or a repelling
Fatou petal of $f$. We would like to emphasize that
 this is not a rigorous method since we do not know how many terms of the series (\ref{EqnPsi}) to take to
 obtain the desired precision. However, empirically,
 the asymptotic expansion (\ref{EqnPsi}) approximates the Fatou
 coordinates $\Psi_a(z),\Psi_b(z)$ with a very high precision.
 For instance, in \cite{LY} Lanford and Yampolsky used
the series $(\ref{EqnPsi})$ in their computational scheme for the
fixed point of the parabolic renormalization operator. Our work in
progress \cite{DS} give us a reason to hope that the described
method can be made rigorous.


\begin{thebibliography}{}


\bibitem{Asp} M. Aspenberg, {\it The Collet-Eckmann
condition for rational functions on the Riemann sphere},
Ph.D. Thesis, KTH, Sweden, 2004.

\bibitem{AM} A. Avila, C.G. Moreira, {\it Statistical
properties of unimodal maps: the quadratic family}, Annals of
Mathematics, 161, no. 2 (2005), 831-881.

\bibitem{BBY06} I. Binder, M. Braverman, M. Yampolsky, {\it
 On computational complexity of Siegel Julia sets}, Commun. Math.
 Phys. {\bf 264} (2006), no. 2, 317-334

\bibitem{BBY07} I. Binder, M. Braverman, M. Yampolsky, {\it Filled
Julia sets with empty interior are computable}, Journ. of FoCM, {\bf
7} (2007), 405-416.

\bibitem{BBY09} I. Binder, M. Braverman, M. Yampolsky, {\it
 Constructing locally connected non-computable Julia sets}, Commun. Math.
 Phys. {\bf 291} (2009), 513-532.

\bibitem{Bra04} M. Braverman, {\it Computational complexity of Euclidian sets: Hyperbolic Julia sets are poly-time computable},
Master's thesis, University of Toronto, 2004.

\bibitem{Bra06} M. Braverman, {\it Parabolic Julia sets are polynomial time computable}, Nonlinearity {\bf 19}, (2006), no.6, 1383-1401.

\bibitem{BY06} M. Braverman, M. Yampolsky, {\it Non-computable Julia
sets}, Journ. Amer. Math. Soc. {\bf 19} (2006), no. 3, 551-578.

\bibitem{BY08} M. Braverman, M. Yampolsky, {\it Computability of Julia sets},
Series: Algorithms and Computation in Mathematics, Vol. 23, Springer, 2008.

\bibitem{Con} J. B. Conway, {\it Functions of one complex variable
$II$}, Springer-Verlag, New York, New York, 1995.

\bibitem{DS} A. Dudko, D. Sauzin, {\it Ecalle-Voronin invariants via resurgence and mould calculus}, in
preparation.

\bibitem{Eca1}
J. \'Ecalle.
\newblock {\em Les fonctions r\'esurgentes, Vol.~1}.
\newblock Publ. Math. d'Orsay 81-05, 1981.

\bibitem{Eca2}
J. \'Ecalle.
\newblock {\em Les fonctions r\'esurgentes, Vol.~2}.
\newblock Publ. Math. d'Orsay 81-06, 1981.

\bibitem{GS} J. Graczyk, G. Swiatek, {\it Generic hyperbolicity in the logistic family}, Ann. of Math., v.
146 (1997), 1-52.

\bibitem{LY} Oscar Lanford III, M. Yampolsky, {\it The fixed point of the parabolic renormalization operator},
 arXiv:1108.2801.

\bibitem{L} M. Lyubich, {\it Dynamics of quadratic polynomials}, I-II. Acta Math., 178 (1997), 185-297.

\bibitem{Mane} R. Ma\~{n}\'{e}, {\it On a Theorem of Fatou}, Bol. Soc. Bras. Mat., Vol. 24, no. 1, 1993, pp. 1-11.

\bibitem{M} J. Milnor, {\it Dynamics in one complex variable. Introductory lectures}, 3rd ed., Princeton University Press, 2006.

\bibitem{Pap} C. M. Papadimitriou, {\it Computational complexity}, Addision-Wesley, Reading, Massachusetts, 1994.

\bibitem{Ram} J.-P. Ramis. S\'{e}ries divergentes et th\'{e}ories asymptotiques, Panoramas et
Synth\`{e}ses (1994).

\bibitem{Ret05} R. Rettinger, {\it A fast algorithm for Julia sets of hyperbolic rational functions}, Electr. Notes Theor. Comput. Sci. {\bf 120} (2005), 145-157.

\bibitem{Sa} D. Sauzin, {\it Resurgent functions and splitting problems}, RIMS Kokyuroku 1493, (2005), 48-117,
available from http://www.imcce.fr/Equipes/ASD/person/Sauzin/sauz preprint.php

\bibitem{Shi} M. Shishikura, Tan Lei, {\it An alternative proof of Mane's theorem on non-expanding Julia sets},
 in ``The Mandelbrot set, Theme and Variations'', Ed. Tan Lei, London Math. Soc. Lect. Note Ser. 274, Cambridge Univ.
Press, 2000, p.265-279.

\bibitem{Sip} M. Sipser, {\it Introduction to the theory of computation, second edition}, BWS Publishing Company, Boston, 2005.

\bibitem{Wey} H. Weyl, {\it Randbemerkungen zu Hauptproblemen der
Mathematik, II, Fundamentalsatz der Algebra and Grundlagen der
Mathematik}, Math. Z. {\bf 20}(1924), 131-151.

\end{thebibliography}
\end{document}